\documentclass[12pt,reqno]{amsart}

\usepackage{amssymb, amsmath, amsthm,geometry}
\usepackage[lithuanian,english]{babel}
\usepackage{geometry}

\usepackage{bm}
\usepackage{subfig,tikz}
\usepackage{wrapfig}
\usepackage{color}
\usepackage{xspace}
\usetikzlibrary{shapes.geometric}

\usepackage{graphicx}
\usepackage[colorlinks=true, citecolor=blue, linkcolor=red, urlcolor=blue]{hyperref}
\newgeometry{asymmetric, centering}
 \usepackage{xcolor}
\usepackage{hyperref}

\newtheorem{thm}{Theorem}[section]
\newtheorem{hypothesis}[thm]{Hypothesis}
\newtheorem{lem}[thm]{Lemma}
\newtheorem{cor}[thm]{Corollary}
\newtheorem{prop}[thm]{Proposition}
\newtheorem{defn}[thm]{Definition}

\numberwithin{equation}{section}

\newcommand{\bel}{\begin{equation} \label}
\newcommand{\ee}{\end{equation}}
\def\beq{\begin{equation}}
\def\eeq{\end{equation}}
\newcommand{\bea}{\begin{eqnarray}}
\newcommand{\eea}{\end{eqnarray}}
\newcommand{\beas}{\begin{eqnarray*}}
\newcommand{\eeas}{\end{eqnarray*}}
\newcommand{\pd}{\partial}

\newcommand{\R}{\mathbb{R}}
\newcommand{\A}{\mathcal{A}}
\newcommand{\C}{\mathbb{C}} 
\newcommand{\CI}{\mathcal{C}}
\newcommand{\N}{\mathbb{N}}

\newcommand{\M}{\mathcal{M}}
\newcommand{\g}{\bar{g}}

\newcommand{\Sym}{\mathrm{Sym}}

\renewcommand{\div}{\mathrm{div}\,}  

\newcommand{\supp}{\mathrm{supp}\,}  

\allowdisplaybreaks

\def\epsilon{\varepsilon}
\def\phi {\varphi}

\renewcommand{\leq}{\leqslant}
\renewcommand{\geq}{\geqslant}

\providecommand{\norm}[1]{\left\lVert#1\right\rVert}

\DeclareMathOperator{\Tr}{\textit{Tr}}

\numberwithin{equation}{section}

\title[Hyperbolic Dirichlet to Neumann Map]{Recovery of time dependent coefficients from Boundary Data for hyperbolic equations}

\author[]{Ali Feizmohammadi}
\address{Department of Mathematics, University College London, London, UK-WC1E  6BT, United Kingdom}
\email{a.feizmohammadi@ucl.ac.uk}
\author[]{Joonas Ilmavirta}
\address{Department of Mathematics and Statistics, University of Jyv\"{a}skyl\"{a}, P.O. Box 35 (MaD), Finland}
\email{joonas.ilmavirta@jyu.fi}
\author[]{Yavar Kian}
\address{Aix Marseille Universit\'{e}, Universit\'{e} de Toulon, CNRS, CPT, Marseille, France}
\email{yavar.kian@univ-amu.fr}
\author[]{Lauri Oksanen}
\address{Department of Mathematics, University College London, London, UK-WC1E  6BT, United Kingdom}
\email{l.oksanen@ucl.ac.uk}

\begin{document}
\maketitle
\begin{abstract}
We study uniqueness of the recovery of a time-dependent magnetic vector-valued potential and an electric scalar-valued potential on a Riemannian manifold from the knowledge of the Dirichlet to Neumann map of a hyperbolic equation. The Cauchy data is observed on time-like parts of the space-time boundary and uniqueness is proved up to the natural gauge for the problem. The proof is based on Gaussian beams and inversion of the light ray transform on Lorentzian manifolds under the assumptions that the Lorentzian manifold is a product of a Riemannian manifold with a time interval and that the geodesic ray transform is invertible on the Riemannian manifold. 
\end{abstract}

\section{Introduction}
\subsection{Statement of the problem}

Let $(\M,\bar{g})$ be a $1+n$ dimensional Lorentzian manifold with boundary. Throughout this paper, we will assume that $(\M,\g)$ has a global product structure, that is to say $\M=[0,T]\times M$, $\g=-dt^2+g$, where $T>0$ and $(M,g)$ denotes a smooth compact connected Riemannian manifold of dimension $n \geq 2$, with smooth boundary $\partial M$. We assume that $g \in \mathcal C^6(M;\Sym^2 M)$, where $\Sym^2 M$ denotes the bundle of symmetric two-tensors over $M$. We denote by $\Delta_{\g}$ the Laplace-Beltrami operator given by $\Delta_{\g} :=  \div_{\g}\, \, \nabla^{\g}$,
where $\div_{\g}$ (resp., $\nabla^{\g}$) denotes the divergence (resp., gradient) operator on $(\M,\g)$.
In local coordinates $(x^0:=t,x^1,\ldots,x^n)$ and for each $u \in C^{2}(\M)$ we have:
$$ \Delta_{\g} u =\sum_{i,j=0}^n  |\g|^{-1/2}\partial_{x_i}( |\g|^{1/2}{\g}^{ij}\partial_{x_j} u),$$
where $\bar{g}^{-1} :=(\bar{g}^{ij})_{0 \leq i,j \leq n}$ and $|\g| :=|\mbox{det}\ \g|$. As $\g=-dt^2+g$, we have $\Delta_{\g}=-\pd_t^2+\Delta_g$, where $\Delta_g$ is defined analogously. Consider a complex-valued scalar function $q(t,x)$ (electric potential) and a complex-valued one-form $\A$ (magnetic potential), such that the following regularity assumptions are satisfied
\bel{regularity}
q \in \CI(\M) \quad\text{and} \quad \A \in \CI^1(\M;T^*\M).
\ee
In local coordinates, the one-form $\A$ can be expressed as 
\bel{Alocal}\A(t,x)=b(t,x) \,dt+\sum_{i=1}^n\omega_j(t,x)\,dx^j= b(t,x)\,dt+\omega(t,x),\ee
where $\omega$ is a time-dependent one-form on $(M,g)$. Given $\A$ and $q$ as above, We consider the initial boundary value problem (IBVP)
\bel{eq1}
\left\{ \begin{array}{ll}  -\Delta_{\g} u + \A\nabla^{\g} u+ q u  =  0, & \mbox{on}\ \M,\\  
u  =  h, & \mbox{on}\ (0,T)\times \partial M,\\  
 u(0,\cdot)  =0,\quad  \pd_tu(0,\cdot)  =0 & \mbox{on}\ M,
\end{array} \right.
\ee
with non-homogeneous Dirichlet data $ h\in H^1_0((0,T)\times \pd M)$. Note that in local coordinates, $\mathcal{A}\nabla^{\g}u=-b\,\pd_t u+\sum_{i,j=1}^n g^{ij}\omega_i\, \pd_j u$. We introduce $\bar{\nu}$, the outward unit normal vector to $(0,T)\times\partial M$, and define the hyperbolic Dirichlet-to-Neumann (DN in short) map, given by 
\[\Lambda_{\A,q}:h\mapsto (\partial_{\bar{\nu}} u-\frac{(\A\bar{\nu})u}{2})_{|(0,T)\times\partial M}\]
where $u$ solves problem \eqref{eq1}. This is well-defined as equation \eqref{eq1} admits a unique solution 
$u \in \mathcal C(0,T;H^1(M)) \cap \mathcal C^1(0,T;L^2(M)),$ 
with $\pd_{\bar\nu} u|_{(0,T)\times \pd M} \in L^2((0,T)\times \pd M)$. This follows from \cite[Theorem 2.1]{LLT} as explained in Section~\ref{directpr}. The goal of this paper is to study the unique recovery of the complex valued coefficients $\A$ and $q$ given $\Lambda_{\A,q}$, up to the natural obstructions discussed in the next section.
\subsection{Natural obstructions}

The first obstruction concerns the recovery of the magnetic potential $\A$. Indeed, for $j=1,2$, fix $(\A_j,q_j)$ defined as above and assume that there exists $\psi(t,x) \in \mathcal C^2(\M)$ with $\psi|_{(0,T)\times \pd M}=0$, such that
\bel{gauge}\A_1=\A_2+2\bar{d}\psi,\quad q_1=q_2+\Delta_{\g}\psi -\A_2\nabla^{\g}\psi-\left\langle \nabla^{\g}\psi,\nabla^{\g}\psi\right\rangle_{\g},\ee
where $\bar{d}$ denotes the exterior derivative on $\M$, that acts on $f \in C^{\infty}(\M)$ through $\bar{d}f =\pd_t f \,dt + \,df$ with $d$ denoting the exterior derivative on $M$, and $\langle \cdot,\cdot\rangle_{\g}$ denotes the inner product on $(\M,\g)$. Then, for $u_j$ the solution of \eqref{eq1} with $\A=\A_j$, $q=q_j$, $j=1,2$, it holds that $u_1=e^{\psi} u_2$ and, using the fact that $\psi_{|(0,T)\times\partial M}=0$, we obtain 
$$\partial_{\bar{\nu}} u_1-\frac{(\A_1\bar{\nu})u_1}{2}=\partial_{\bar{\nu}} u_2+(\partial_{\bar{\nu}} \psi)u_2-\frac{(\A_1\bar{\nu})u_2}{2}=\partial_{\bar{\nu}} u_2-\frac{(\A_2\bar{\nu})u_2}{2}.$$
This proves that $\Lambda_{\A_1,q_1}=\Lambda_{\A_2,q_2}$, but $\A_1 \neq \A_2$ as soon as $\psi$ does not vanish identically. In other words, the DN map $\Lambda_{\A,q}$ is invariant with respect to the gauge transformation given by \eqref{gauge} and the best we can expect is the recovery of the coefficients $\A$ and $q$ from $\Lambda_{\A,q}$ modulo the gauge invariance \eqref{gauge}. 

The second obstruction to our problem is due to finite speed of propagation for the wave equation. Let us define the set  
$$ \mathcal D=\{(t,x)\in \M\,|\, 
\textrm{dist}(x,\partial M)< t < T - \textrm{dist}(x,\partial M)\}.$$
Due to domain of dependence arguments (see \cite[Section 1.1]{Ki2} and \cite[Section 1.1]{KiOk}), it is not possible to recover the restriction of any of the coefficients $\A$ and $q$ on the set $\M \setminus \mathcal D$ from $\Lambda_{\A,q}$. Therefore, for our problem, the best we can expect, is to recover the coefficients modulo the gauge invariance above, on the set $\mathcal D$. 

\subsection{Main result}
\label{mainresult}
Before stating the main result, we need to recall the definition of the geodesic ray transform on the (spatial) Riemannian manifold $M$. For each $(y,v) \in SM$, with $SM$ denoting the unit tangent bundle of $M$, let $\gamma(\cdot;y,v)$ denote the unit speed geodesic starting at point $y$, in the direction $v$, that is:
$$ \nabla^g_{\dot{\gamma}}\dot{\gamma}(\cdot;y,v)=0, \quad \gamma(0;y,v)=y, \quad \dot{\gamma}(0;y,v)=v.$$
For any $(y,v) \in SM$, we define $\tau_{exit}(y,v)$ through:
$$\tau_{exit}(y,v):=\inf{\{t>0\,|\,\gamma(t;y,v) \in \pd M,\, \dot{\gamma}(t;y,v) \notin T_{\gamma(t;y,v)}\pd M\}}.$$
We now define
\bel{Gamma-}
\pd_-SM:= \{ (y,v) \in SM\,|\, y \in \pd M,\,  \langle v,\nu(y)\rangle_g < 0,\, \tau_{exit}(y,v)<\infty\},
\ee 
where $\nu$ denotes the outward normal unit vector on $\pd M$. Henceforth, for the sake of brevity, we use the term maximal geodesic to refer to the geodesics $\gamma(\cdot;y,v)$ (or $\gamma(\cdot)$ in short) with $(y,v) \in \pd_-SM$, over their maximal interval of definition in $M^{int}$, that is the interval $I:=(0,\tau_{exit})$.

\begin{defn}
\label{geodesicraydef}
Let $(y,v) \in \pd_-SM$ and let $\gamma(\cdot;y,v):I\to M$. We define the geodesic ray transform of $(f,\alpha) \in \CI(M) \times \CI(M;T^*M)$, as follows:
$$  \mathcal{I}_{\gamma}(f,\alpha):= \int_{I} [f(\gamma(t))+ \alpha(\gamma(t))\dot{\gamma}(t)]\,dt.$$
\end{defn}

\noindent We also need to recall the definition of the solenoidal component, $\alpha^s$, of a one-form $\alpha$ with local representation $\alpha=\sum_{k=1}^n \alpha_k \,dx^k$. Let $\delta \alpha:=\sum_{i,j=1}^n\frac{1}{\sqrt{g}}\pd_i(\sqrt{g}g^{ij}\alpha_j)$ denote the divergence operator on $M$ sending one-forms to functions. Given any $\alpha \in L^2(M;T^*M)$, it can be uniquely decomposed as
\bel{}
\alpha=\alpha^s + d\psi,
\ee
where $\delta \alpha^s =0$ and $\psi \in H^1_0(M)$ solves $\Delta_g \psi =\delta \alpha$ in the weak sense on $M$. This is called the Helmholtz decomposition (see for example \cite{Taylor}). We will be working with $\CI^1(M)$ one-forms. In this case, one can immediately see that since $\delta \alpha \in \CI(M) \subset L^{n}(M)$, elliptic regularity implies that $\psi \in W^{2,n}(M) \subset \CI^1(M)$. We will use this observation later in the paper to derive the smoothness properties for the gauge. With these notations, we can now state the main geometric assumption on $(M,g)$. 
\begin{hypothesis}
\label{ginjectivity}
Let $(M,g)$ be a compact connected Riemannian manifold with smooth boundary. We say that the geodesic ray transform is injective on $M$ with respect to functions $f \in \CI(M)$ and one-forms $\alpha \in \CI(M;T^*M)$, if the following holds:
$$ \mathcal{I}_{\gamma(\cdot;y,v)} (f,\alpha) =0, \quad \forall (y,v) \in \pd_-SM \quad \text{implies that} \quad f \equiv 0 \quad \text{and} \quad \alpha^s \equiv 0.$$ 
\end{hypothesis}

According to Theorems 3 and 4 in \cite{SU1}, Hypothesis \ref{ginjectivity} will be fulfilled if $M$ is simple. This condition can also be fulfilled by a non-simple manifold. We refer to Section~\ref{sectiongr} for a more detailed discussion about this aspect. 

Finally, let us introduce the set $\mathcal E \subset \mathcal D$ where we recover the coefficients. For any $x \in M$, we define $D_g(x):=\sup{\{\tau_{exit}(y,v)\,|\,(y,v) \in \pd_-SM,\,\text{$x$ is in $\gamma(\cdot;y,v)$}\}}$ and let $D_g(M):=\sup{\{D_g(x)\,|\,x \in M\}}$. For $T>2D_g(M)$, we define
$$ \mathcal E:=\{(t,x)\in \M\,|\, D_g(x)< t<T-D_g(x)\}.$$

\begin{thm}\label{t1} Let $g \in \CI^6(M;\Sym^2 M)$, $\A_1, \A_2\in \mathcal C^1(\M;T^*\M)$ and $q_1,q_2\in \CI(\M)$. Assume that $\supp{(\A_1-\A_2)}\subset \mathcal{E}$, $\supp{(q_1-q_2)} \subset \mathcal{E}$,  
and  
$$\A_1(t,x)=\A_2(t,x),\quad \forall(t,x)\in(0,T)\times \pd M.$$
If Hypothesis~\ref{ginjectivity} holds, then $\Lambda_{\A_1,q_1}=\Lambda_{\A_2,q_2}$
implies that there exists $\psi \in \mathcal C^2(\M)$ with $\psi|_{(0,T)\times \pd M}=0$, such that \eqref{gauge} holds.
\end{thm}
\bigskip

The proof of this theorem relies in part on injectivity of the so-called light ray transforms of vector valued functions. Recall that a curve $\beta$ in $\M$ is a null geodesic (or a light ray), if $\nabla^{\bar{g}}_{\dot{\beta}} \dot{\beta}=0$ and $\langle \dot{\beta},\dot{\beta}\rangle_{\g}=0$. Given the product structure of $\M$, we can parametrize maximal null geodesics through $\beta(t)=(\tilde{t}+t,\gamma(t)), t \in I$ where $\gamma$ is a unit speed maximal geodesic in $M$ and $\tilde{t} \in \R$. We define the light ray transforms $\mathcal L_q$ and $\mathcal L_{\mathcal A}$ of $q \in \CI(\M)$ and $\mathcal A \in \CI(\M;T^*\M)$ as follows:
$$\mathcal L_{q} (\beta):= \int_I q(\beta(t))\,dt \quad \text{and}\quad \mathcal L_{\mathcal A}(\beta):=\int_\beta \mathcal A .$$  
We have the following proposition, that will be proved in Section~\ref{inversionsection}.
\begin{prop}
\label{introlightray}
Let $(q,\mathcal \A) \in \CI(\M) \times \CI^1(\M;T^*\M)$ be such that $\supp{q},\supp{\mathcal A} \subset \mathcal E$. If Hypothesis~\ref{ginjectivity} holds, then we have the following statements:
\begin{itemize}
\item[(i)] If $\mathcal L_q (\beta)=0$ for all maximal null geodesics $\beta$ in $\M$, then $q \equiv 0$ on $\M$.
\item[(ii)]If $\mathcal L_{\mathcal A}(\beta)=0$ for all maximal null geodesics $\beta$ in $\M$, then there exists $\psi \in \CI^2(\M)$ vanishing on $\pd \M$ such that $\mathcal A\equiv\bar{d}\psi$ on $\M$.
\end{itemize}
\end{prop}

\subsection{Previous literature}
Results related to the recovery of coefficients for hyperbolic equations can in general be divided into two categories of time-independent and time-dependent coefficients. Starting with the seminal works \cite{Bel87,Bel92}, there is an extensive literature related to the recovery of time-independent coefficients for hyperbolic equations. These results usually rely on the Boundary Control method, developed in \cite{Bel87,Bel92} and a time sharp unique continuation theorem \cite{Tataru}, which provide the building blocks of very general results. We refer the reader to \cite{KOM} for an introduction to the method and to \cite{KOP} for an example of a state of the art result in this direction. We also refer to \cite{Bel,KKL} for review. The unique continuation theorem in \cite{Tataru} fails if the dependence of the coefficients on time is non-analytic and therefore extension of these results for general time dependent coefficients is not possible (see e.g. \cite{Al,AB}). We refer the reader to \cite{E2} for a uniqueness result, when the dependence of the coefficients on time is real analytic. Starting with \cite{BK}, methods based on Carleman estimates have also been quite fruitful in deriving uniqueness results for time-independent coefficients of hyperbolic equations. Contrary to the Boundary Control method, where the best known stability estimates are double logarithmic \cite{BKL}, these methods tend to give strong stability estimates. We also mention \cite{LOk} where Boundary Control method is combined with complex geometric optics and stronger estimates are obtained for low frequencies.

In the time-dependent category, most of the results are obtained for the reconstruction of the zeroth order term, $q$, and are based on a use of geometric optic solutions for the wave equation. Let us mention that this approach has also been used in the time-independent category \cite{BD,BJY,Ki,SU2} to obtain strong stability estimates although they suffer from considerably stronger geometric assumptions compared to the Boundary Control method. Methods based on geometric optics were used in the context of recovery of time-dependent coefficients starting with \cite{St}. Among the literature of results in this direction, we refer to \cite{A,I,Ki3,RS,S}. The leading coefficients for the wave equation in all these results are constant. Uniqueness of zeroth order coefficient $q$ for a variable coefficient wave equation was considered in \cite{KiOk} where the recovery of the potential was based on inversion of geodesic ray transform for scalar functions. It should be noted that even in the case where $\A=0$, the result in this paper is a significant improvement of \cite{KiOk}, since there $(M,g)$ was assumed to be simple.

Approaches based on global geometric optic solutions fail if the Riemannian manifold $(M,g)$ is not simple (see Section~\ref{geometricinterpretation}). This motivates the use of Gaussian beams in the current paper. Gaussian beams were introduced in \cite{BU,Ra} and they were first used in the context of inverse problems in \cite{BeKa,KaKu}. We refer the reader to \cite{KKL} for a thorough presentation in the case of a wave equation with a smooth metric, a smooth electric potential and no magnetic potential. This paper is concerned with the reconstruction of time-dependent vector valued coefficients for the wave equation under weak geometrical assumptions on the spatial manifold $(M,g)$ and weaker regularity assumptions on the coefficients. We use Gaussian beams to reduce the inverse problem to the inversion of the light ray transform of the unknown coefficients. The closest previous work to this reduction is \cite{SY}, where the authors study the problem of recovery of the geometry along with a time-dependent magnetic potential $\A$ and an electric potential $q$ in a Lorentzian manifold from a micro-local formulation of a Cauchy data set on the boundary. It is shown that if $\g,\A,q$ belong to some $\CI^k$, with $k$ sufficiently large, then this Cauchy data set uniquely determines the scattering relation of $\g$ along with light ray transforms of $\A,q$. Their approach is based on the study of Fourier Integral Operators and propagation of singularities. The inversion of the light ray transform on a general Lorentzian manifold is left as an open problem. Our Gaussian beam construction makes the reduction to the light ray transform more explicit in terms of the smoothness required. We also succeed in the inversion of the light lay transform, in the sense of Proposition~\ref{introlightray}. Our inversion method for the light ray transform was inspired in part by techniques developed in the context of the Calder\'{o}n problem \cite{Cekic}. 
\subsection{Outline of the paper}
This paper is organized as follows. In Section~\ref{prem}, we discuss the forward problem \eqref{eq1} and also discuss the Hypothesis~\ref{ginjectivity}. In Section~\ref{gaussian section}, we present the Gaussian beam construction near a null geodesic in $\M$. In Section~\ref{reductionsection}, we show the reduction step from the knowledge of the Dirichlet to Neumann map $\Lambda_{\A,q}$, to the knowledge of the light ray transforms of $\A,q$ and conclude that Theorem~\ref{t1} follows from Proposition~\ref{introlightray}. Finally, Section~\ref{inversionsection} is concerned with the proof of Proposition~\ref{introlightray}.    

\section{Preliminaries}
\label{prem}

\subsection{Direct problem}
\label{directpr}
Let $Xu:=\A\nabla^{\g}u+qu$, where $\A$ and $q$ satisfy \eqref{regularity}. We consider the wave equation
    \begin{align}\label{wave}
\begin{cases}
-\Delta_{\g} u + X u = F, & \text{in $\M$,}
\\
u|_{x \in \pd M} = h, & \text{on $(0,T) \times \pd M$,}
\\
u|_{t=0} = u_0, \quad \pd_t u|_{t=0} = u_1, & \text{on $M$}
\end{cases}
    \end{align}
It is classical that $u$ is in the energy space
    \begin{align}\label{energy_sp}
\CI([0,T]; H^1(M)) \cap \CI^1([0,T]; L^2(M))
    \end{align}
when $h = 0$, $F \in L^2(\M)$, $u_0 \in H^1_0(M)$ and $u_1 \in L^2(M)$.
The wave equation
    \begin{align}\label{wave0}
\begin{cases}
-\Delta_{\g} v = F, & \text{in $\M$,}
\\
v|_{x \in \pd M} = h, & \text{on $(0,T) \times \pd M$,}
\\
v|_{t=0} = u_0, \quad \pd_t v|_{t=0} = u_1, & \text{on $M$,}
\end{cases}
    \end{align}
was considered in \cite{LLT}. It was shown there that if $F$ and $u_1$ are as above, and $u_0 \in H^1(M)$ and $h \in H^1((0,T) \times \pd M)$ satisfy the compatibility condition 
    \begin{align}\label{comp_cond}
h|_{t=0} = u_0|_{x\in \pd M},
    \end{align}
then the solution $v$ is the energy space (\ref{energy_sp}), and 
$\pd_{\bar\nu} v|_{x \in \pd M} \in L^2((0,T) \times \pd M).$

Let us now set $u = v - w$ where $v$ 
is the solution of (\ref{wave0}) with $F$, $u_0$, $u_1$ and $h$ as above, and $w$ is the solution of (\ref{wave}) with $F= Xv \in L^2(\M)$, $u_0 = 0$, $u_1 = 0$ and $h=0$. Then $u$ satisfies (\ref{wave}) with the same $F$, $u_0$, $u_1$ and $h$ as in (\ref{wave0}) for $v$. As both $v$ and $w$ are in (\ref{energy_sp}), so is $u$. But then $-\Delta_{\g} u = F - X u \in L^2(\M)$ and $\pd_\nu u|_{x \in \pd M} \in L^2((0,T) \times \pd M)$.
It is straightforward to turn this regularity result to the corresponding estimate
    \begin{align}\label{energy}
&\norm{u}_{\CI([0,T]; H^1_0(M)) \cap \CI^1([0,T]; L^2(M))}
+ \norm{\pd_\nu u}_{L^2((0,T) \times \pd M)}
\\\notag&\quad\le C (\norm{F}_{L^2(\M)} + \norm{h}_{H^1((0,T) \times \pd M)} + \norm{u_0}_{H^1(M)} + \norm{u_1}_{L^2(M)}),
    \end{align}
for solutions $u$ of (\ref{wave}) under the compatibility condition (\ref{comp_cond}).

\def\I{I}
We write $\I f(t) = \int_0^t f(s) ds$, and show now that $u$ satisfies the estimate
    \begin{align}\label{energy_int}
\norm{u}_{L^2(\M)} \le C \norm{\I F}_{L^2(\M)},
    \end{align}
when $u_0$, $u_1$ and $h$ vanish identically.  
The map taking $F$ to $u$ is continuous from $L^2(\M)$ to $H^1(\M)$, and by considering its adjoint, we obtain also continuity from $H^{-1}(\M)$ to $L^2(\M)$.
Let $F \in L^2(\M)$ and define $v$ and $w$ as above, but with $u_0$, $u_1$ and $h$ vanishing also for $v$. Then $-\Delta_{\g} I v = I F$ and
    \begin{align*}
\norm{v}_{L^2(\M)} = 
\norm{\pd_t I v}_{L^2(\M)}
\le C \norm{I F}_{L^2(\M)}.
    \end{align*}
Moreover,
    \begin{align*}
\norm{w}_{L^2(\M)} \le 
C \norm{X v}_{H^{-1}(\M)}
\le C \norm{v}_{L^2(\M)},
    \end{align*}
and the above two estimates imply (\ref{energy_int}).

\subsection{On inversion of the geodesic ray transform}
\label{sectiongr}
We are not aware of results on inversion of the geodesic ray transform $\mathcal I_\gamma(f,\alpha)$ when the metric tensor $g$ is assumed to be only $\mathcal C^6$-smooth. Therefore, for the purposes of the present section, we assume $g$ to be $\mathcal C^\infty$-smooth.

It is well-known that simple manifolds satisfy Hypothesis~\ref{ginjectivity}, see Theorems 3 and 4 in \cite{SU1}.
It is also likely that the method to invert the geodesic ray transform using convex foliations, originating from \cite{UV}, can be used to show that Hypothesis~\ref{ginjectivity} holds under the assumptions that the boundary of $M$ is strictly convex in the sense of the second fundamental form, and that there is a strictly convex function on $M$. In \cite{PSUZ} this is shown under the further assumption that $f$ and $\alpha$ are smooth.
Let us point out that, even when $\mathcal A = 0$, combining Theorem~\ref{t1} and \cite{PSUZ} gives a result on unique determination of smooth $q$ that falls outside the scope of the closest previous results \cite{KiOk}.

Let us now describe a non-simple case satisfying Hypothesis~\ref{ginjectivity} as studied in \cite{SU_08}.
There it is assumed that $(M,g)$ satisfies the following: 
\begin{itemize}
\item[(i)] $M$ and $\pd M$ have real analytic atlases.
\item[(ii)] There is an open set of simple geodesics $\Gamma$ on a slightly larger manifold $(\hat M, g)$ such that $T^*M \subset \{N^* \gamma\,|\, \gamma \in \Gamma\}$.
\item[(iii)] Any path in $M$ connecting two boundary points is homotopic to a path of the form 
$$
c_1 \cup \gamma_1 \cup c_2 \cup \gamma_2 \cup \dots \cup \gamma_k \cup c_{k+1},
$$ 
where $c_j$ are paths on $\pd M$ and $\gamma_j = \tilde \gamma_j|_M$ for some $\tilde \gamma_j \in \Gamma$. Moreover, $\gamma_j$ intersects $\pd M$ only at its endpoints and is transversal to $\pd M$.
\item[(iv)] $g$ is generic in the sense of \cite[Corollary 1]{SU_08}.
\end{itemize}
Here $\gamma \in \Gamma$ being simple means that the endpoints of $\gamma$ are in $\hat M \setminus M$ and there are no conjugate points on $\gamma$, simple geodesics are given topology in the sense of the parametrization (2) in \cite{SU_08}, and $N^* \gamma$ denotes the conormal bundle of $\gamma$, viewed as a 1-dimensional submanifold of $\hat M$.

\section{Gaussian Beam Solutions}
\label{gaussian section}
The goal of this section is to construct the so called Gaussian beam solutions $u_j\in \mathcal C([0,T];H^1( M))\cap \mathcal C^1([0,T];L^2( M))$ for $j=1,2$ of the problems

\bel{Gsol}
\begin{array}{ll}
\left\{\begin{array}{l}
-\Delta_{\g}u_1+\A_1\nabla^{\g}u_1+q_1u_1=0,\quad (t,x)\in \M,\\ 
u_1(0,x)=\partial_tu_1(0,x)=0,\quad  x\in M,\end{array}\right.
\\
\\
\left\{\begin{array}{l}
-\Delta_{\g}u_2-\A_2\nabla^{\g}u_2+(-\bar \delta \mathcal A_2+q_2)u_2=0,\quad (t,x)\in \M,\\ 
u_2(T,x)=\partial_tu_2(T,x)=0,\quad  x\in M,
\end{array}\right.
\end{array}
\ee
taking the form
\bel{go1}u_1(t,x)=e^{ i\rho \phi(t,x)}v_1(t,x)+R_{1,\rho}(t,x),\quad (t,x)\in \M,\ee
\bel{go2}u_2(t,x)=e^{ -i\rho \bar{\phi}(t,x)}\bar{v}_2(t,x)+R_{2,\rho}(t,x),\quad (t,x)\in \M,\ee
with $\rho>1$. Here, $\bar \delta$ denotes the divergence operator on $(\M,\g)$ sending one-forms to functions. The two equations in \eqref{Gsol} are in essence formal adjoints of each other, with respect to the real $L^2(\M)$ inner product.   
The phase function $\phi$ will be chosen so that both the oscillatory parts $e^{i\rho \phi}$ and $e^{-i\rho\bar{\phi}}$ remain bounded in $L^2(\M)$ as $\rho \to \infty$ and such that the principal terms $|e^{ i\rho \phi(t,x)}v_k(t,x)|$ are concentrated near a fixed maximal null geodesic in $\mathcal D$. The remainder terms $R_{1,\rho}, R_{2,\rho}$ will vanish in the limit $\rho \to \infty$. 


\subsection{Fermi Coordinates}

We will start by reviewing Fermi coordinates near a fixed maximal null geodesic $\beta: [\tau_-,\tau_+] \to \M$, where we are using the time coordinate $t$ as the parametrization for the null geodesic. Here, $\beta(\tau_-), \beta(\tau_+)$ denote the start and end points of the maximal null geodesic on the boundary $(0,T) \times \partial M$. Note, in particular, that $\beta(t) \in \mathcal{D}$ for all $t \in [\tau_-,\tau_+]$. Fermi coordinates were first introduced by E. Fermi \cite{F}. In this paper, the geometry has a product structure which makes the construction of Fermi coordinates slightly easier. This is to some extent similar to \cite{DKLS,KSa}, where a coordinate construction was carried out in the context of an elliptic partial differential equation on a Riemannian manifold with a product structure. We will therefore follow \cite{DKLS} with some modifications. 

Let us introduce notation that will be fixed throughout the remainder of this paper. We begin by embedding $(M,g)$ into a closed manifold $(\hat{M},g)$ and extend the null geodesic $\beta$ to $\hat{\M}:=(0,T) \times \hat{M}$ such that $\beta(t)$ is well-defined on the interval $[\tau_- -\epsilon,\tau_++\epsilon]$ with $\epsilon>0$ a small constant. We also consider extensions of the metric $g$ and the coefficients $\A_1,\A_2,q_1,q_2$ to the bigger set $\hat{\M}$ such that the extended metric $g$ is $\CI^6$ smooth and the extensions of the coefficients satisfy the regularity assumptions in \eqref{regularity} with $\M$ replaced by $\hat{\M}$. Finally, for the sake of convenience, we define the constants $a,b,a_0,b_0,s_-$ and $s_+$ as follows 
\bel{normalizedconstants}
a=\sqrt 2(\tau_--\epsilon), \quad b=\sqrt 2(\tau_++\epsilon), \quad a_0= \sqrt 2(\tau_--\frac{\epsilon}{2}), \quad b_0=\sqrt 2(\tau_++\frac{\epsilon}{2}),
\ee
and
$$ s_-=\sqrt{2} \tau_-, \quad s_+=\sqrt 2\tau_+.$$
%
\noindent We will now present Fermi coordinates near the null geodesic $\beta$ in $\hat{\M}$. In all the following arguments, $\beta(t)$ denotes the parametrization of the null geodesic with respect to the time coordinate in $\M$. 
\begin{lem}(Fermi coordinates)
\label{fermi}
Let $\beta:(\tau_--\epsilon,\tau_++\epsilon) \to \hat{\M}$ be a null geodesic as above. There exists a coordinate neighborhood  $(U,\Phi)$ of $\beta([\tau_--\frac{\epsilon}{2},\tau_++\frac{\epsilon}{2}])$ denoted by $(z^0:=s,z^1:=r,z^2,\ldots,z^n)$ such that:
\begin{itemize}
\item {$\Phi(U)=(a,b) \times B(0,\delta')$ where $a,b$ are given by \eqref{normalizedconstants} and $B(0,\delta')$ denotes a ball in $\mathbb{R}^{n}$ with a sufficiently small radius $\delta'$ only depending on the geometry $(\M,\g)$ and $\epsilon$.}
\item{$\Phi(\beta(t))=(\sqrt{2}t,\underbrace{0,\ldots,0}_{n \hspace{1mm}\text{times}})$ for all $t \in (\tau_--\epsilon,\tau_++\epsilon)$.} 

\end{itemize}
Furthermore the metric $\g$ in this coordinate system is $\CI^4$ smooth and satisfies 
$$\bar{g}|_\beta = 2dsdr+\sum_{j=2}^n (dz^j)^2 \quad \text{and}\quad \frac{\pd\g_{jk}}{\partial z^i}|_\beta = 0\quad \text{for all}\quad 0\leq i,j,k \leq n.$$
\end{lem}

\begin{proof}
Let us begin by defining $\Pi: \hat{\M} \to \hat{M}$ through $\Pi(t,x)=x$. Note that $\Pi \beta = \gamma$ where $\gamma$ is a unit speed geodesic passing through the point $x_0=\Pi(\beta(a))$. We choose $\{\alpha_2,...,\alpha_n\}$
such that the set $\{\dot \gamma(x_0),\alpha_2,...,\alpha_n\}$ forms an orthonormal basis for $T_{x_0}\hat{M}$. Let $y_1$ denote the arc length parameter along the geodesic $\gamma$ from the point $x_0$. For each $2\leq k \leq n$, let $e_k(y_1) \in T_{\gamma(y_1)}\hat{M}$ denote the parallel transport of $\alpha_k$ along $\gamma$ to the point $\gamma(y_1)$. Since $\dot{\gamma}$ is also parallel along $\gamma$, the set $\{\dot \gamma(y_1),e_2(y_1),...,e_n(y_1)\}$ forms an orthonormal basis for $T_{\gamma(y_1)}\hat{M}$. We now define the coordinate system $(y^0,\ldots,y^n)$ through $\mathcal F_1:\mathbb{R}^{n+1} \to \hat{\M}$:
$$\mathcal F_1(y^0:=t,y^1,...,y^n) = (t,\exp_{\gamma(y^1)}(\sum_{\alpha=2}^{n}y^{\alpha} e_{\alpha}(y^1))),$$
where $\exp_{p}(\cdot)$ denotes the exponential map on $M$ at the point $p$. Let us remark that since $g \in \CI^6(M;\Sym^2M)$, the map $\mathcal F_1$ is locally in $\CI^5$ (see for example \cite{DK}). 

We now define $(s:=z^0,r:=z^1,\ldots,z^n)=\mathcal F_2 (y^0,\ldots,y^n)$ through 
\bel{}
\left\{\begin{array}{l}     
 s :=z^0:=\frac{1}{\sqrt{2}}(t+y^1)+\frac{a}{2},\\ 
r:=z^1:=\frac{1}{\sqrt{2}}(-t+y^1)+\frac{a}{2},\\
z^j := y^j  \quad \forall j \geq 2. \end{array}\right.\ee
For the sake of brevity, we will also use the notations $z=(s,z')=(s,r,z'')$ for this coordinate system. Let us consider the composition map $\mathcal F:\R^{1+n} \to \hat{\M}$ given by $\mathcal F= \mathcal F_1 \circ \mathcal F_2^{-1}$. It is clear that for all $t \in [\tau_--\epsilon,\tau_++\epsilon]$:
$$\mathcal F(\sqrt{2}t,0) =\mathcal{F}_1(t,t-a\frac{\sqrt{2}}{2},0)=\mathcal{F}_1(t,t-(\tau_--\epsilon),0)=\beta(t),$$  
implying that $\mathcal{F}(s,0)$ is injective for all $s \in (a,b)$ as $\beta$ is not self-intersecting on the time interval $[\tau_--\epsilon,\tau_++\epsilon]$. Furthermore, for all $s \in (a,b)$ it holds that 
\[
\begin{aligned}
&\frac{\partial}{\partial s} \mathcal F(s,0)= \frac{1}{\sqrt 2}(\pd_t +\dot{\gamma}(\frac{\sqrt 2}{2}(s-a))),\\
&\frac{\partial}{\partial r}\mathcal F(s,0)=\frac{1}{\sqrt 2}(-\pd_t+\dot{\gamma}(\frac{\sqrt 2}{2}(s-a))),\\
&\frac{\partial}{\partial \tau} \mathcal F(s,\tau v_{\alpha})|_{\tau=0}=e_{\alpha}(\frac{\sqrt 2}{2}(s-a)),
\end{aligned}
\]
where $v_{\alpha}$ denotes the $\alpha$th coordinate vector in $\mathbb{R}^{n-2}$.  Thus $\mathcal F(s,z')$ is a locally $\mathcal C^5$ map in a neighborhood of the null geodesic $\beta$ such that $\mathcal F(s,0)$ is injective and $D\mathcal{F}(s,0)$ is invertible. The inverse mapping theorem applies to deduce that $\mathcal F(s,z')$ is a diffeomorphism on a neighborhood of $[a_0,b_0]\times B(0,\delta')$ with $\delta'$ sufficiently small. We then choose $\Phi = \mathcal F^{-1}$. Note that $\Phi \in \CI^5$ near the null geodesic, and since $g \in \CI^6(M;\Sym^2M)$, we deduce that the pull back of the metric metric $\g$, $\Phi^*\g$ is $\CI^4$ smooth in the Fermi coordinate system. 

Let us now study the form of the metric in this single coordinate chart $(U,\Phi)$ given by the $z$-coordinates. We will first derive the form of the metric in $y$-coordinates which is just an affine transformation of the $z$ coordinates (the linear part of this affine transformation is unitary). To find the form of the metric in $y$-coordninates, we note that $\mathcal F_1$ preseves the product structure on $\hat{M}$ and therefore it suffices to check the form of the Riemannian metric $g$ near the geodesic $\gamma=\Pi(\beta)$ in $\hat{M}$. Let the indices $i,j,k$ run between $1$ and $n$ and the indices $\alpha,\beta$ between $2$ and $n$.  Since the set $\{\dot \gamma(y_1),e_2(y_1),...,e_n(y_1)\}$ is an orthonormal basis, we see that $g_{jk}|_{\gamma} = \delta_{jk}|_{\gamma}$. This implies that $\partial_1 g_{jk}|_{\gamma}=0$. Now note that:
$$ \partial_{\alpha} g_{ij}|_\gamma = \langle \nabla^g_{\partial_\alpha} \partial_i, \partial_j \rangle_g|_\gamma + \langle \partial_i, \nabla^g_{\partial_\alpha} \partial_j \rangle_g|_\gamma,$$
\noindent where $\nabla^g$ denotes the Levi-Civita connection on $(\hat{M},g)$. Using the symmetry for Levi-Civita connection we see that
$$ \nabla^g_{\partial_\alpha} \partial_1 |_\gamma= \nabla^g_{\partial_1} \partial_\alpha|_\gamma = \nabla^g_{\dot{\gamma}(y^1)} e_{\alpha}(y^1)|_\gamma = 0,$$
which together with $\nabla^g_{\dot{\gamma}}\dot{\gamma}=0$ implies that $\partial_j g_{11}|_{\gamma} =0$ for all $j \in \{1,\ldots,n\}$. This implies that $\Gamma^1_{1j}|_\gamma=0$ for all $j$, where $\Gamma^i_{jk}$ denotes the Christoffel symbol for $(\hat{M},g)$. Pick an arbitrary unit vector $ (v^2,\ldots, v^n) \in \mathbb{R}^{n-1}$ and for each $y^1 \in \mathbb{R}$ consider the geodesic in $\hat{M}$ parametrized as $h(\tilde{r})=\exp_{\gamma(y_1)}(\tilde{r} \sum_{\alpha=2}^n v^{\alpha}e_\alpha(y_1))$ with the corresponding Fermi coordinates $(y_1,\tilde{r}v^2,\ldots,\tilde{r}v^n)$. Note that $\dot{h}^{\alpha}(0)=v^{\alpha}$, and $h$ satisfies
$$ 0=\ddot{h}^{k}_{\alpha \beta}(\tilde{r}) + \Gamma^{k}_{\alpha \beta}(\tilde{r})\dot{h}^\alpha(\tilde{r}) \dot{h}^\beta(\tilde{r}) =\Gamma^{k}_{\alpha \beta}(\tilde{r})\dot{h}^\alpha(\tilde{r}) \dot{h}^\beta(\tilde{r}),$$
since $\ddot{h}^k_{\alpha \beta}(\tilde{r})=0$ in this coordinate system. Since $v^{\alpha}$ is arbitrary, we deduce that 
$$ \Gamma^k_{\alpha \beta}|_\gamma=0 \quad \text{for all $1\leq k\leq n$ and all $2\leq \alpha,\beta \leq n$}.$$
\noindent To conclude that all the Christoffel symbols vanish on the null geodesic, we still need to show $\Gamma^{\alpha}_{1\beta}|_{\gamma}=0$. Using the definition of the Christoffel symbol we see that it suffices to show that $\pd_\alpha g_{1\beta}|_\gamma=0$. But,
$$ \pd_\alpha g_{1\beta}|_\gamma=\langle \pd_1, \nabla^g \pd_\alpha\pd_\beta\rangle_g|_\gamma=\Gamma^1_{\alpha\beta}|_\gamma=0.$$ 
 Thus in $y$ coordinates $\g|_{\beta}=\eta_{jk}$ where $\eta_{jk}$ denotes the Minkowski metric on $\R^{1+n}$ and $\pd_{i}\g_{jk}|_{\beta}=0$. Since the map $ y \to z$ is affine, it is easy to verify that $\g$ satisfies the claimed properties. 
\end{proof}

It should be clear now that the constants $a,a_0,s_-,s_+,b_0,b$ defined in \eqref{normalizedconstants} merely denote the $s$-coordinates of the points $\beta(\tau_--\epsilon)$, $\beta(\tau_--\frac{\epsilon}{2})$, $\beta(\tau_-)$, $\beta(\tau_+)$, $\beta(\tau_++\frac{\epsilon}{2})$, and $\beta(\tau_++\epsilon)$ respectively.

\subsection{Eikonal and Transport equations}
\label{eiktrans}


\noindent Throughout this subsection we will assume that $\beta$ is a null geodesic in $\mathcal D$ that is extended to $\hat{\M}$ as described above with local coordinates $(z^0,\ldots,z^n)$ and that $\A,q$ satisfy \eqref{regularity}. Let us consider the differential operator:
$$L_{\A,q} := -\Delta_{\g} + \A \nabla^{\g} + q.$$


\noindent We start the construction of the approximate Gaussian beam by defining the set
\bel{setV}
\mathcal{V}=\{(z^0,z') \in \hat{\M}\,|\,z^0\in [a_0,b_0],|z'|<\delta\},
\ee
with $0<\delta<\delta'$ (see Lemma~\ref{fermi}) sufficiently small such that the set $\mathcal V$ does not intersect the sets $\{0\}\times M$ and $\{T\}\times M$ (this can always be fulfilled as $\beta \in \mathcal D$). We make a WKB ansatz of the form 
\bel{WKB}u = e^{i\rho \phi} v,\ee
such that $u$ is an approximate solution to $L_{\A,q} u =0$ as a formal power series in $\rho>1$. Here, $\phi \in \CI^3(\mathcal V)$ and $v\in\CI^2(\mathcal V)$. We make the following ansatz for $\phi, v$ respectively:
\bel{phiv}
\phi = \sum_{k=0}^{2} \phi_k(s,z') \quad \text{and} \quad v(s,z')= \rho^{\frac{n}{4}} v_0(s) \chi(\frac{|z'|}{\delta}),
\ee
where for each $k=0,1,2$, $\phi_k$ is a homogeneous polynomial of degree $k$ with respect to the variables $z^{i}$ with $i \in\{1,...,n\}$. The smooth function $\chi:\R \to [0,\infty]$ satisfies $\chi(t)=1$ for $|t| \leq \frac{1}{4}$, and $\chi(t)=0$ for $|t| \geq \frac{1}{2}$. We also define the set where 

\noindent Note that:
$$\Delta_{\g} (e^{i\rho \phi} v) = e^{i\rho \phi} (-\rho^2 \langle d\phi, d\phi \rangle_{\bar{g}} v + i \rho ( 2 \langle d\phi,dv \rangle_{\bar{g}} + (\Delta_{\g} \phi)v) + \Delta_{\g} v).$$
Let us define:
$$\mathcal{S}\phi:=\langle d\phi, d\phi \rangle_{\bar{g}},\quad \text{and}\quad\mathcal{T}_{\A}v:= 2 \langle d\phi,dv \rangle_{\bar{g}}+ (-\A\nabla^{\g} \phi +\Delta_{\g} \phi)v.$$
Then:
\bel{conjugate}
\begin{aligned}
 L_{\A,q} (e^{i\rho \phi} v) &= e^{i\rho \phi} (\rho^2 (\mathcal{S}\phi) v - i \rho \mathcal{T}_{\A} v + L_{\A,q} v).\\
\end{aligned}
\ee
We require that $\mathcal S\phi=\sum_{k,l=0}^n \bar{g}^{kl} \partial_k \phi \, \partial_l \phi $ vanishes up to second order on the null geodesic $\beta$ with respect to the transversal directions (the case $\mathcal S \phi \equiv 0$ is known as an eikonal equation). Put differently, in terms of the Fermi coordinates we require that:
\bel{eikonal}
\frac{\partial^{\alpha_1}}{\partial {z_1}^{\alpha_1}}\ldots \frac{\partial^{\alpha_n}}{\partial {z_n}^{\alpha_n}} (\mathcal{S}\phi)(s,0,\ldots,0) = 0 \quad \text{for} \quad  s \in (a_0,b_0), 
\ee
for all $m=0,1,2$ and all choices of $\alpha_1,\ldots, \alpha_n \geq 0$ such that $\sum_{j=1}^{n} \alpha_j = m$. \\
\noindent We will also require that the following transport equation holds along the null geodesic $\beta$:
\bel{trans}
  (\mathcal{T}_{\A}v_0)(s,0,\ldots,0)= 0 \quad \text{for} \quad  s \in (a_0,b_0). 
\ee
\subsection{Construction of the phase}
\label{phaseconst}
\noindent We begin by solving equation \eqref{eikonal}. For $m=0$, we obtain the equation
$$ (\sum_{k,l=0}^n\bar{g}^{kl} \frac{\partial \phi}{\partial z_k} \frac{\partial \phi}{\partial z_l})|_{\beta} = 0 \quad 1 \leq \forall i,j \leq n.$$
Recalling that $\bar{g}|_\beta = 2dz^0dz^1+(dz^2)^2+\ldots+(dz^n)^2$, this reduces to 
\bel{0th order}
2 \partial_0 \phi \, \partial_1 \phi + \sum_{k=2}^n (\partial_k \phi)^2 =0.
\ee
Similarly, for $m=1$, we obtain (recall that for all $i,j,k$ we have $\partial_i g^{jk}|_\beta=0$) 
\bel{1st order}
(\sum_{k,l=0}^n\bar{g}^{kl} \partial^2_{k\alpha} \phi \, \partial_l \phi)|_\beta=0 ,\quad \text{for all $1\leq \alpha \leq n$}.
\ee
\noindent \noindent Recalling the definition of the phase function $\phi$ from equation \eqref{phiv}, it is clear that equations \eqref{0th order} and \eqref{1st order} will be satisfied if we set $\phi_0=0$ and $\phi_1=r$.  Next we consider the case $m=2$ in equation \eqref{eikonal} and write $\phi_2(s,z') := \sum_{1 \leq i,j \leq n} H_{ij}(s) z^{i}z^{j}$ where $H_{ij}=H_{ji}$ is a symmetric matrix. Let us impose the auxiliary condition that
\bel{postivity}
\Im {H(s)}>0\quad \text{ for $s \in (a_0,b_0)$}.
\ee  
This assumption will lead to a Gaussian decay away from the null geodesic $\beta$, but we will also provide a geometric motivation behind this assumption in the next section. We require:
$$ \frac{\partial^2}{\partial z_i \partial z_j} (\sum_{k,l=0}^n\bar{g}^{kl} \frac{\partial \phi}{\partial z_k} \frac{\partial \phi}{\partial z_l})|_{\beta} = 0 \quad 1 \leq \forall i,j \leq n.$$
This is equivalent to:
$$ (\partial^2_{ij}\bar{g}^{kl} \partial_k \phi \, \partial_l \phi+ 2 \bar{g}^{kl} \partial^3_{kij} \phi \, \partial_l \phi +2\bar{g}^{kl} \partial^2_{ki} \phi \, \partial^2_{lj}\phi + 4 \pd_i \g^{kl}\pd_{jk}^2\phi\,\pd_l \phi )|_{\beta}=0.$$
\noindent which again simplifies to
$$(\partial^2_{ij}\bar{g}^{11}+2\bar{g}^{10} \partial^3_{0ij}\phi +2\sum_{k=2}^{n} \partial^2_{ki} \phi \, \partial^2_{kj}\phi )|_{\beta}=0.$$
We therefore obtain the following Riccati type ODE:  
\bel{riccati}
\frac{d}{ds} H + HCH + D=0, \quad s \in (a_0,b_0) \quad H(s_{-})=H_0\quad \text{with}\quad \Im H_0>0,
\ee
where $C$ is the matrix defined through
\bel{Cmatrix}
\left\{ \begin{array}{rcll} C_{11}= 0&\\
C_{ii}=2& \quad 2\leq  i \leq n \\
 C_{ij}=0& \quad \text{otherwise} \end{array}\right.
\ee 
and $D=(D_{ij})_{n \times n}:=\frac{1}{4}(\pd^2_{ij} {\g}^{11}|_\beta)_{n \times n}$ for all $i,j \in \{1,\ldots,n\}$. Note that since $g$ is $\CI^4$ smooth in the Fermi coordinates, we have $D \in \CI^2([a_0,b_0];\C^{n\times n})$. We now recall two lemmas. For the proofs, we refer the reader to \cite[Lemma 8,Section 8]{KKL} and  \cite[Lemma 10,Section 8]{KKL} respectively.
\begin{lem}
\label{ricA}
The Riccati equation (\ref{riccati}) has a unique solution. The solution $H$ is symmetric and $\Im(H(s))>0$ for all $s \in (a_0,b_0)$. We have $H(s)=Z(s)Y(s)^{-1}$ where 
$Z(t)$ and $Y(t)$ solve the following system of first order linear ODEs:
$$\frac{d}{ds}Y = C Z, \quad Y(s_-)=I,$$
$$\frac{d}{ds}Z= -D Y, \quad Z(s_-)=H_0.$$
In addition, $Y(s)$ is non-degenerate  for all $s \in [a_0,b_0]$.
\end{lem}
\begin{lem}
\label{ricB}
The following identity is satisfied:
$$\det(\Im(H(s))\cdot|\det(Y(s))|^2=\det(\Im(H_0)).$$
\end{lem}
\noindent Let us make some remarks about the regularity of the solutions $Y(s),H(s)$. Since $C$ is a constant matrix, the matrix $Y(s)$ also satisfies
\bel{Ymatrix}
\frac{d^2}{ds^2} Y=-CD Y,\quad Y(s_-)=I,\quad \dot{Y}(s_-)=CH_0.  
\ee
since $D \in \CI^2([a_0,b_0];\C^{2n})$, we immediately deduce that $Y \in \CI^4([a_0,b_0];\C^{n\times n})$. Now considering the ODE for the function $Z(s)$, we deduce that $Z \in \CI^3([a_0,b_0];\C^{n \times n})$. Finally, since $H(s)=Z(s)Y^{-1}(s)$ and since $Y(s)$ is non-singular on $[a_0,b_0]$, we conclude that
\bel{YHreg}
H \in \CI^3([a_0,b_0];\C^{n \times n})\quad \text{and}\quad \phi \in \CI^3(\mathcal V)
\ee  
in the Fermi coordinates.
\subsection{Construction of the amplitude}
\noindent Let us now study the transport equation \eqref{trans}. First observe that in Fermi coordinates
$$(\Delta_{\g} \phi)|_{\beta}=\sum_{i,j=0}^N \g^{ij}\pd^2_{ij}\phi|_\beta=\sum_{j=2}^{n}\pd^2_{jj}\phi|_\beta=\Tr(CH).$$
 Thus equation \eqref{trans} simplifies to:
\bel{trans11}
2 \partial_s v_0 + (\Tr(CH)-\A(s,0)\dot{\beta})v_0=0\quad s\in[a_0,b_0],
\ee
where $\A\dot{\beta}:= \A \pd_s=\langle \A, dz^1 \rangle_{\g}.$ 
\noindent We proceed to prove that
\bel{transport}
v_0(s)= \det(Y(s))^{-\frac{1}{2}} e^{\frac{1}{2}(\int_{s_{-}}^s \A(\tau,0)\dot\beta \,d\tau)} \quad s \in [a_0,b_0]
\ee
satisfies equation \eqref{trans11}. Indeed, this follows immediately from the observation:
$$\Tr(C(s)H(s))=\Tr(C(s)Z(s)Y(s)^{-1})=\Tr(\frac{dY}{ds}(s) Y(s)^{-1}) = \frac{d}{ds} \log(\det(Y(s))),$$
where we have used the fact that $\frac{dY}{ds}(s)=C(s)Z(s)$. Clearly $v_0 \in \CI^2([a_0,b_0])$, which together with the definition of $v$ implies that $v \in \CI^2(\mathcal V)$. This concludes the construction of the amplitude function and also the construction of $u$ defined by \eqref{WKB} which we refer to as an approximate Gaussian beam.
\subsection{Geometrical interpretations}
\label{geometricinterpretation}
We will briefly discuss some geometrical aspects of the approximate Gaussian beam construction. In particular, we will discuss explicitly, how conjugate points on $(\M,\g)$ manifest themselves in the vector valued function $Y(s)$ constructed above. First we note the following lemma.
\begin{lem}
\label{curvature}
Let $\beta$ be a null geodesic as above and let us consider the Fermi coordinates near $\beta$. We have the following identity:
$$ \frac{\pd^2 {\g}^{11}}{\pd z_i \pd z_j}|_{\beta} =-2\mathcal{R}_{0i0j}|_{\beta},$$
for all indices $i,j \in \{1,\ldots,n\}$, where $\mathcal{R}_{\lambda\mu\nu\kappa}$ denotes the curvature tensor for the Lorentzian manifold $(\hat{\M},\g)$. 
\end{lem}
\begin{proof}
First note that by Bianchi identities the expression is clearly symmetric with respect to indices $i,j$. We let $\nabla^{\g}$ and $\bar{\Gamma}^{i}_{jk}$ denote the Levi-Civita connection and the Christoffel symbol for $(\M,\g)$ respectively. Recall that $\g^{ij}$ denotes the inverse of the matrix $\g_{ij}$ and since $\g|_{\beta}=2dz^0dz^1+(dz^2)^2+\ldots+(dz^n)^2$ it follows that
$$ \frac{\pd^2 {\g}^{11}}{\pd z_i \pd z_j}|_{\beta} =-\frac{\pd^2 {\g}_{00}}{\pd z_i \pd z_j}|_{\beta}.$$
Since $[\pd_i,\pd_j]=0$, we have
$$\mathcal{R}_{0i0j}|_\beta=\langle \nabla^{\g}_{\pd_0}\nabla^{\g}_{\pd_i} \pd_0, \pd_j \rangle_{\g}|_\beta -\langle \nabla^{\g}_{\pd_i}\nabla^{\g}_{\pd_0} \pd_0, \pd_j \rangle_{\g}|_\beta.$$
Using the definition of the Levi-Civita connection we have 
\bel{curvature1}
\nabla^{\g}_{\pd_0}\nabla^{\g}_{\pd_i} \pd_0= \sum_{k=0}^n(\pd_0 \bar{\Gamma}^k_{i0}) \pd_k +\sum_{k,m=0}^n \bar{\Gamma}^k_{i0} \bar{\Gamma}^{m}_{0k} \pd_m,
\ee
and
\bel{curvature2}
\nabla^{\g}_{\pd_i}\nabla^{\g}_{\pd_0} \pd_0= \sum_{k=0}^n(\pd_i \bar{\Gamma}^k_{00}) \pd_k +\sum_{k,m=0}^n \bar{\Gamma}^k_{00} \bar{\Gamma}^{m}_{ik} \pd_m.
\ee
Recall that $\pd_{i}g_{jk}|_{\beta}=0$. This implies that $\bar{\Gamma}^i_{jk}|_{\beta}=0$ but since $\pd_0$ also denotes the tangent vector to $\beta$, we observe additionally that $\pd_{0} \Gamma^i_{jk}|_\beta=0$ which implies that \eqref{curvature1} vanishes along $\beta$. Hence we obtain
\[ 
\begin{aligned}
\mathcal{R}_{0i0j}|_\beta&= -\langle \sum_{k=0}^n \pd_i \bar{\Gamma}^k_{00}\pd_k, \pd_j \rangle_{\g}|_\beta=-(\sum_{k=0}^n \bar{g}_{jk}\pd_i \bar{\Gamma}^{k}_{00})|_\beta=\frac{1}{2}(\sum_{k,m=0}^n \g_{jk}\g^{km}\pd^2_{im}\g_{00})|_\beta\\
&=\frac{1}{2}(\sum_{m=0}^n \delta_{jm}\pd^2_{im}\g_{00})|_\beta=\frac{1}{2}\pd^2_{ij}\g_{00}|_\beta=-\frac{1}{2}\pd^2_{ij}\g^{11}|_\beta.
\end{aligned}
\]
\end{proof}
\noindent We can next use the above lemma in conjunction with the product structure of the Lorentzian manifold $\M$ to derive the following corollary.
\begin{cor}
\label{product} 
For any null geodesic $\beta$ as above, we have the following 
$$ \frac{\pd^2 \g^{11}}{\pd z_1 \pd z_i}|_{\beta} =0,$$
for all indices $i \in \{1,\ldots,n\}$. 
\end{cor}
\noindent This corollary can be used to simplify the Riccati equation \eqref{riccati} further. Indeed, Corollary~\ref{product} implies that $D_{1i}=0$ for $i\in \{1,\ldots,n\}$ and since $C_{1j}=C_{j1}=0$ for all $j \in \{1,\ldots,n\}$, we can simply take $H_{11}=s+c_0$ for any constant $c_0$ with $\Im(c_0)>0$, $H_{1j}=H_{j1}=0$ for all $j>1$ and take $H_{i+1,j+1}:=\tilde{H}_{i,j}$ for all $i,j \in \{1,\ldots,n-1\}$ where $\tilde{H}$ is a symmetric $(n-1)\times(n-1)$ matrix satisfying
\bel{riccati1}
\frac{d}{ds} \tilde{H} + 2\tilde{H}^2 + \tilde{D}=0, \quad \tilde{H}(s_{-})=\tilde{H}_0,
\ee 
with  $\tilde{D}_{ij}=D_{i+1,j+1}$ for all $i,j \in \{1,\ldots,n-1\}$. This observation also simplifies the construction of the matrix $Y$. Indeed, we can take $Y_{11}=c_0$ and $Y_{1j}=Y_{j1}=0$ for all $j \in \{1,\ldots,n\}$. 

Note that a key ingredient in the construction of Gaussian beams is the requirement that the matrix valued function $Y(s)$ is non-singular. This is indeed guaranteed in the above construction as a consequence of choosing $\Im(H_0)>0$. We will briefly discuss what happens when one pursues real valued solutions to this linear system. Recall that $Y(s)$ satisfies equation \eqref{Ymatrix}.
This of course implies that the columns of the matrix $Y$ should also satisfy the same ODE. Let $V$ be one of the columns of $Y$ with representation $V=\sum_{j=1}^n V^j(s)\frac{\pd}{\pd z_j}$. Using the definition of the matrix $D$ and Lemma~\ref{curvature}, we deduce that $\frac{d^2}{ds^2} V^i=\sum_{j=1}^n \mathcal{R}_{0i0j}V^j$. Rearranging the indices and using the Bianchi identities we obtain (recall that on the null geodeisc, $\pd_0=\dot{\beta}$):
$$ \frac{D^2}{ds^2} V+\mathcal{R}(V,\dot{\beta})\dot{\beta}=0.$$
This is the well-known Jacobi equation along $\beta$. We therefore see that the columns of $Y$ are variation fields of some variation of $\beta$ through null geodesics. In particular based on this geometric characterization of $Y$, one can deduce that if there exists a point $\beta(s)$ on the interval $[a_0,b_0]$ that is conjugate to $\beta(a_0)$, then any real valued solution $Y(s)$ to \eqref{Ymatrix} will always become singular at that point (see for example \cite[Section 5.5]{Car92}). Therefore a global geometric optic construction with a real valued phase function can not be achieved in the presence of conjugate points on $(\M,\g)$.


\subsection{Construction of the remainder terms}

With the WKB construction complete, we now return to the task of constructing solutions $u_1,u_2$ to \eqref{Gsol}, concentrating on a null geodesic $\beta \in \mathcal{D}$. In particular, we will construct the remainder terms in equations \eqref{go1}-\eqref{go2}.  We consider the differential operators $L_{\A_1,q_1}$ and $L^*_{\A_2,q_2}$ (formal adjoint of $L_{\A_2,q_2}$ with respect to the real $L^2$ inner product). We can use the previous discussion to obtain two families of approximate solutions given by $e^{i\rho \phi} v_{1}$ and $e^{-i \rho \overline{\phi}} \bar{v}_{2}$ to these differential operators. Indeed, let
$$ F_{1,\rho} = - L_{\A_1,q_1}(e^{i\rho \phi}v_1), \quad \quad F_{2,\rho}=-L^*_{\A_2,q_2}(e^{-i\rho \bar{\phi}}\bar{v_2}).$$
\noindent Applying equation (\ref{conjugate}), we obtain
\bel{f1}
\begin{aligned}
F_{1,\rho}&=-e^{i \rho \phi}[\rho^2 (\mathcal{S}\phi) v_1 -i \rho \mathcal{T}_{\A_1}v_1+L_{\A_1,q_1}v_1],\\
F_{2,\rho}&=-e^{-i \rho \bar{\phi}}[\rho^2 (\overline{\mathcal{S}\phi}) \bar{v}_{2}+i \rho \overline{\mathcal{T}_{-\bar{\A}_{2}}v_2}+L^*_{\A_2,q_2}\bar{v}_{2}].
\end{aligned}
\ee
The phase function $\phi \in \mathcal C^3(\mathcal V)$ is chosen exactly as in Section~\ref{phaseconst} and adapting equation (\ref{trans}) to this case, we make the following ansatz for $v_1,v_2$:
$$v_i = \rho^{\frac{n}{4}} v_{i,0}(s) \chi(\frac{|z'|}{\delta}) \quad \text{for $ i=1,2$},$$
\noindent such the functions $v_{i,0}(s)$ satisfy the following transport equations:
\bel{trans1}
\begin{aligned}
2 \partial_s v_{1,0} + (\Tr{(CH)}-\A_{1}\dot{\beta})v_{1,0}&=0,\quad s \in [a_0,b_0] \\
2 \partial_s v_{2,0} + (\Tr{(CH)}+\bar{\A}_{2}\dot{\beta})v_{2,0}&=0,\quad s \in [a_0,b_0].
\end{aligned}
\ee 
\noindent Using \eqref{transport} we have:
\bel{trans2}
\begin{aligned}
v_{1,0}(s)&=\det(Y(s))^{-\frac{1}{2}} e^{\frac{1}{2}(\int_{s_-}^s (\A_{1}\dot{\beta})(\tau,0) \,d\tau)},\\
v_{2,0}(s)&=\det(Y(s))^{-\frac{1}{2}} e^{-\frac{1}{2}(\int_{s_-}^s (\bar{\A}_{2}\dot{\beta})(\tau,0) \,d\tau)}.
\end{aligned}
\ee

Note that $F_{1,\rho},F_{2,\rho}$ are compactly supported in a small tubular region around the null geodesic where the Fermi coordinates are well defined. Also recall from the previous discussions that $v_1,v_2 \in \CI^2(\mathcal V)$ in the Fermi coordinates. 
\noindent Next we define the expression $R_{j,\rho}$, $j=1,2$, as the solution of the following IBVP
\bel{eqgo1}\left\{ \begin{array}{rcll} L_{\A_1,q_1}R_{1,\rho} & =  F_{1,\rho},&  (t,x) \in \M ,\\
R_{1,\rho}(0,x)=0,\ \ \partial_t R_{1,\rho}(0,x)&=0,& x\in M\\
 R_{1,\rho}(t,x)=0,& \  & (t,x) \in (0,T)\times \pd M ,& \end{array}\right.
\ee
\bel{eqgo2}\left\{ \begin{array}{rcll} L_{\A_2,q_2}^*R_{2,\rho} & =  F_{2,\rho},&  (t,x) \in \M ,\\
R_{2,\rho}(T,x)=0,\ \ \partial_t R_{2,\rho}(T,x)&=0,& x\in M\\
 R_{2,\rho}(t,x)=0,& \  & (t,x) \in (0,T)\times \pd M .& \end{array}\right.
\ee

\noindent The energy estimate \eqref{energy} in Section~\ref{prem} implies that equations \eqref{eqgo1} and \eqref{eqgo2} admit unique solutions 
$$R_{j,\rho}\in  \mathcal C([0,T];H^1_0( M))\cap \mathcal C^1([0,T];L^2( M))\quad j=1,2,$$ 
with the estimates:
\bel{energyremainder}
\|R_{j,\rho}\|_{H^1(\M)} \leq C \|F_{j,\rho}\|_{L^2(\M)}.
\ee
We claim that $R_{j,\rho}$, $j=1,2$, satisfy the following decay property
\bel{GO2} \lim_{\rho\to+\infty}(\norm{R_{j,\rho}}_{L^2(\M)}+\rho^{-1}\norm{R_{j,\rho}}_{H^1(\M)})=0,\ee

\noindent and showing this completes the construction of the solutions $u_1,u_2$ of \eqref{Gsol}. Note that for $j=1,2$, using \eqref{eikonal}, \eqref{trans} we have the following bounds:
\bel{ansatz bound}
\begin{aligned}
&\|v_j\|_{\mathcal C^2(\M)} \leq C \rho^{\frac{n}{4}},\\
&| \mathcal{T}_{\A_{j}} v_j| \leq C  \rho^{\frac{n}{4}}|z'|\chi(\frac{|z'|}{\delta}),\\
&|\mathcal{S}\phi| \leq C |z'|^2 \mathcal{N}(|z'|) \quad \text{for $z \in \mathcal V$},
\end{aligned}
\ee  
\noindent where $C>0$ only depending on the geometry and $\|\A_1\|_{\CI^1},\|\A_2\|_{\CI^1}$. Here, $\mathcal{N}$ denotes a continuous function depending on the geometry $(\hat{\M},\g)$ and such that $\mathcal{N}(0)=0$. Note that equation \eqref{postivity} implies that:
\bel{quad}
|e^{i\rho \phi}|=|e^{-i\rho \bar{\phi}}| \leq e^{-D \rho |z'|^2},
\ee
\noindent with $D>0$ independent of $\rho$ and only depending on the geometry. In particular these estimates imply that:
\bel{roughbounds}
\begin{aligned}
\|\rho^{\frac{n}{4}} e^{i\rho \phi}\|^2_{L^2(\mathcal V)} &\lesssim \int_{\mathcal{V}}  \rho^{\frac{n}{2}} e^{-D\rho|z'|^2}\chi^2(\frac{|z'|}{\delta})dz = O(1),\\
\|\rho^{\frac{n}{4}}(\mathcal{S}\phi) e^{i\rho \phi}\|^2_{L^2(\mathcal V)} &\lesssim \int_{\mathcal{V}}   |z'|^4 \rho^{\frac{n}{2}}\mathcal{N}^2(|z'|) e^{-D\rho|z'|^2}\chi^2(\frac{|z'|}{\delta})dz = o(\rho^{-2}),\\
\| (\mathcal{T}_{\A_{j}} v_j) e^{i\rho \phi}\|^2_{L^2(\mathcal V)} &\lesssim \int_{\mathcal{V}} \rho^{\frac{n}{2}} |z'|^2 e^{-D\rho|z'|^2}\chi^2(\frac{|z'|}{\delta})dz = O(\rho^{-1}),
\end{aligned}
\ee
Combining these bounds with \eqref{f1}, we find
\bel{bulknorm}\norm{F_{j,\rho}}_{L^{2}(\M)}= o(\rho),\ j=1,2,\ee

\noindent and using the estimate \eqref{energyremainder}, we deduce that
$$\lim_{\rho\to+\infty}\rho^{-1}\norm{R_{j,\rho}}_{H^1(\M)}\leq C\lim_{\rho\to+\infty}\rho^{-1}\norm{F_{j,\rho}}_{L^2(\M)}=0,\quad j=1,2.$$
Therefore, in order to prove the bound \eqref{GO2}, it only remains to prove that
\bel{GObis}\lim_{\rho\to+\infty}\norm{R_{j,\rho}}_{L^2(\M)}=0,\quad j=1,2.\ee
Let us begin by stating the following two lemmas that will simplify the proof of the estimate \eqref{GObis}.
\begin{lem}
\label{nonzero}
Let $\phi,\mathcal{V}$ be as above. Then, $\pd_t \phi$ does not vanish in $\overline{\mathcal{V}}$. 
\end{lem}
\begin{proof}
Since we are considering the neighborhood $\mathcal{V}$ we may use the Fermi coordinate system $(s,z')$. Recall that in this coordinate system $\phi(z)=z^1+ H_{ij}(z^0)z^iz^j=r+H_{ij}(s)z^iz^j$ where $i,j \in \{1,\ldots,n\}$. Therefore:
$$\sqrt{2} \partial_t \phi = \partial_s \phi - \partial_r \phi =-1 + \sum_{1\leq i,j\leq n} \dot{H}_{ij}(s)z^iz^j + 2 \sum_{i=1}^n H_{i1}(s)z^i,$$ 
which implies that for $|z'|=\sqrt{|z^1|^2+\ldots+|z^n|^2}<\delta$ sufficiently small we have 
\bel{}
 |\partial_t \phi| > \frac{1}{2}.
\ee
\end{proof}
\begin{lem}
\label{keyest}
Let $\phi$ be as above and suppose that $f \in \CI(\mathcal{V})$ with $\supp{f} \subset \mathcal V$. The following estimate holds:
\bel{}
\lim_{\rho \to \infty}\rho^{\frac{n}{4}}\|\int_0^t f(\tau,\cdot) e^{i\rho \phi(\tau,\cdot)}\,d\tau \|_{L^2(\M)} = 0.
\ee 
\end{lem}
\begin{proof}
Note that since $f$ is supported near the null geodesic $\beta$, we may use the Fermi coordinate system $(s,z')$. Define $\zeta_\rho:\R^{n+1}\to\R$ through  
\[ \zeta_\rho(x) = \kappa^{-1} \rho^{\frac{n+1}{4}}\chi(\rho^{\frac{1}{4}}|x|),\]
with $\chi$ defined as in \eqref{phiv} and $\kappa=\|\chi\|_{L^1(\R^n)}$. We define the smooth functions:
$$ f_{\rho}= f * \zeta_\rho \quad (s,z') \in \mathcal{V}. $$
\noindent It is clear that $f_\rho$ is supported near the null geodesic $\beta$ and
\bel{mollifier1}
\lim_{\rho\to\infty}\| f_{\rho} - f\|_{\CI(\M)}=0, \quad \quad  \|f_{\rho}\|_{W^{k,\infty}(\mathcal{M})} \leq C_k \rho^{\frac{k}{4}} \quad \forall k \in \mathbb{N}.
\ee
\noindent We write
$$\|\rho^{\frac{n}{4}}\int_0^t f(\tau,\cdot) e^{i\rho \phi(\tau,\cdot)} \,d\tau \|_{L^2(\M)}\leq I_1+I_2,$$
where 
$$I_1= \|\rho^{\frac{n}{4}}\int_0^t f_\rho(\tau,\cdot) e^{i\rho \phi(\tau,\cdot)} \,d\tau \|_{L^2(\M)},\quad \text{and}\quad I_2=\|\rho^{\frac{n}{4}}\int_0^t (f-f_{\rho})(\tau,\cdot) e^{i\rho \phi(\tau,\cdot)} \,d\tau \|_{L^2(\M)}.$$
For $I_1$, we integrate by part in time using $e^{i\rho\phi}=\frac{-i}{\rho\pd_t\phi}\pd_te^{i\rho\phi}$ and use Lemma~\ref{nonzero} together with the fact that the set $\mathcal V$ does not intersect the sets $\{0\}\times M$ and $\{T\}\times M$ to obtain
$$ I_1 \leq \| \frac{\rho^{\frac{n}{4}}}{\rho \pd_t \phi} f_{\rho} e^{i\rho \phi} \|_{L^2(\M)}+\|\rho^{\frac{n}{4}-1}\int_0^t \pd_\tau(\frac{f_{\rho}}{\pd_\tau\phi}) e^{i\rho \phi} \,d\tau \|_{L^2(\M)} \lesssim \rho^{-\frac{3}{4}},$$
where we are using the first bound in \eqref{roughbounds} together with equation \eqref{mollifier1} to bound the first term by $\rho^{-1}$ and the second term by $\rho^{-\frac{3}{4}}$. For the term $I_2$, we can use the Cauchy Schwarz inequality along with inequalities \eqref{roughbounds} and equation \eqref{mollifier1} to obtain
$$I_2 \leq T \|\rho^{\frac{n}{4}}(f-f_{\rho}) e^{i\rho \phi} \|_{L^2(\M)}\leq T\| f_{\rho} - f\|_{\CI(\M)}\cdot \|\rho^{\frac{n}{4}}e^{i\rho\phi}\|_{L^2(\mathcal V)} \to 0 \quad \text{as} \quad \rho \to \infty.$$
\end{proof}

\begin{proof}[Proof of Estimate~\eqref{GObis}]
The result for $R_{1,\rho}$ and $R_{2,\rho}$ being similar, we will only give a proof for $R_{1,\rho}$. Using the energy estimate \eqref{energy_int} in Section~\ref{prem}, together with the fact that $R_{1,\rho}$ solves equation \eqref{eqgo2}, it suffices to prove that $ \lim_{\rho \to \infty} \| F_{*,\rho}\|_{L^2(\M)} = 0$, where $F_{*,\rho}(t,x):=\int_0^t F_{1,\rho}(\tau,x)\,d\tau$. We have:
\bel{l1n}
F_{*,\rho}(t,x)=I_1(t,x)+I_2(t,x)+I_3(t,x),
\ee
where
\[
\begin{aligned}
I_1(t,x)&= -\int_0^t  e^{ i\rho\phi}\left[\rho^2 (\mathcal{S}\phi)(\tau,x) v_1(\tau,x)\right]\, d\tau \\
I_2(t,x)&=\int_0^t e^{i\rho\phi}\left[i\rho \mathcal{T}_{\A_1} v_1(\tau,x) \right]\, d\tau \\
I_3(t,x)&= -\int_0^t  e^{ i\rho\phi}\left[ (L_{\A_1,q_1})v_1(\tau,x) \right]\, d\tau.
\end{aligned}
\]
\noindent We proceed to bound each of the above integrals using the Fermi coordinates $(s,z')$ around $\beta$. This can be done since each of the above integrands is supported in a small tubular neighborhood of $\beta$. For the first integral, $I_1$, we apply integration by parts to obtain
\bel{I_1} \begin{aligned} I_1(t,x)&=\int_0^t \frac{i\rho^2}{\rho \partial_\tau \phi} (\partial_\tau e^{ i\rho\phi})(\mathcal{S}\phi)(\tau,x) v_1(\tau,x)\, d\tau \\
&= \frac{i\rho e^{i\rho\phi}}{\partial_t \phi} (\mathcal{S}\phi)(t,x) v_1(t,x) +\int_0^t  e^{ i\rho\phi}\left[-i\rho \partial_\tau[(\frac{\mathcal{S}\phi}{\partial_\tau \phi})v_1] \right](\tau,x)\, d\tau\\
&=\frac{i\rho e^{i\rho\phi}}{\partial_t \phi} (\mathcal{S}\phi) v_1 +i\rho \int_0^t e^{i\rho\phi}\frac{\pd^2_\tau \phi}{(\pd_\tau \phi)^2}(\mathcal S\phi)v_1 \,dt- i\rho \int_0^t e^{i\rho \phi}\frac{\pd_\tau[(\mathcal S\phi)v_1]}{\pd_\tau \phi}\,dt\\
&:=S_1(t,x)+S_2(t,x)+S_3(t,x).
\end{aligned}
\ee
\noindent Using \eqref{roughbounds} we deduce that $\|S_1\|_{L^2(\M)} = o(1)$. For $S_2(t,x)$, we first use the Cauchy-Schwarz inequality to observe that
$$\|S_2\|_{L^2(\M)} \leq \rho T \|e^{i\rho\phi}\frac{\pd^2_t \phi}{(\pd_t \phi)^2}(\mathcal S\phi)v_1\|_{L^2(\mathcal V)}.$$
Recalling that $\phi \in \CI^3(\mathcal V)$ and applying the bound \eqref{roughbounds} again, we deduce that $\|S_2\|_{L^2(\M)}=o(1)$. For $S_3$, we integrate by parts again to obtain
$$S_3=-e^{i\rho\phi}\frac{\pd_t[(\mathcal{S}\phi)v_1]}{(\pd_t\phi)^2}+\int_0^t e^{i\rho\phi}\pd_\tau[\frac{\pd_\tau((\mathcal{S}\phi)v_1)}{(\pd_\tau\phi)^2}]\,d\tau:=S_4+S_5.$$
Using equation \eqref{eikonal}, it is clear that $|\pd_t(\mathcal S \phi)|\lesssim |z'|\mathcal N(|z'|)$ and  $|\pd^2_t(\mathcal S \phi)|\lesssim \mathcal N(|z'|)$ on the set $\mathcal{V}$. Using this observation, together with bounds analogous to \eqref{roughbounds}, we deduce that
$\|S_4\|_{L^2(\M)}=o(\rho^{-\frac{1}{2}})$ and $\|S_5\|_{L^2(\M)}=o(1)$.
\noindent Let us now consider the function $I_2(t,x)$. We have:
\bel{I_2} \begin{aligned} I_2(t,x)&=  \int_0^t \frac{-i}{\rho \partial_\tau \phi} \partial_\tau (e^{ i\rho\phi})\left[i\rho \mathcal{T}_{\A_1} v_1(\tau,x)\right]\, d\tau \\
&= \frac{e^{i\rho\phi}}{\partial_t \phi} (\mathcal{T}_{\A_1} v_1)(t,x) -\int_0^t  e^{ i\rho\phi}\left[ \partial_\tau(\frac{\mathcal{T}_{\A_1}v_1}{\partial_\tau \phi}) \right] \, d\tau:=S_6(t,x)+S_7(t,x).
\end{aligned}
\ee
\noindent Using \eqref{roughbounds} again, we deduce that $\|S_6\|_{L^2(\M)} \lesssim \rho^{-\frac{1}{2}}$. For $S_7$, noting that $\partial_t(\frac{\mathcal{T}_{\A_1}v_1}{\partial_t \phi})\in \CI_c(\mathcal V)$, we can use Lemma~\ref{keyest} to show that $\lim_{\rho \to \infty} \|S_7\|_{L^2(\M)}=0$. 
\noindent Finally, for the function $I_3(t,x)$, since $\phi, v_1 \in \CI^2(\mathcal V)$, we have $L_{\A_1,q_1}v_1 \in \CI(\mathcal V)$ and that this function is supported near the null geodesic $\beta$. Using this observation together with Lemma~\ref{keyest}, we conclude that $\lim_{\rho \to \infty} \|I_3\|_{L^2(\M)}=0$. This completes the proof.
\end{proof}

\section{Reduction to the light ray transform}
\label{reductionsection}
In this section we will obtain the light ray transforms of $\A$ and $q$ on null geodesics from the knowledge of the Dirichlet to Neumann map $\Lambda_{\A,q}$ and then use Proposition~\ref{introlightray} to prove Theorem~\ref{t1}. Proposition~\ref{introlightray} will subsequently be proved in the next section.
\subsection{Recovery of the light-ray transform of $\A$}
From now on we fix $q=q_1-q_2$, $\A=\A_1-\A_2$ extended by zero to $(\R\times M)\setminus\M$. We take $\beta$ to be a maximal null geodesic in $\mathcal{D} \subset \M$ and extend it to $\hat{\M}$. For $j=1,2$, we define 
$$u_j\in \mathcal C([0,T];H^1( M))\cap \mathcal C^1([0,T];L^2( M))$$ 
a solution of \eqref{Gsol} taking the form \eqref{go1}-\eqref{go2} with the properties described in the previous section. We fix $u_3$ the solution of \eqref{eq1} with $\A=\A_2$, $q=q_2$ and $f$ given by
$$f(t,x):=u_1(t,x)=v_1(t,x)e^{i\rho \phi},\quad (t,x)\in(0,T)\times\pd M.$$ Then the function $u=u_3-u_1$ solves
\bel{eq3}
\left\{ \begin{array}{ll}  -\Delta_{\g} u +\A_2(t,x)\nabla^{\g}u+ q_2(t,x) u  =  \A \nabla^{\g}u_1+qu_1, & \mbox{in}\ \M,\\  
u  =  0, & \mbox{on}\ (0,T)\times \partial M,\\  
 u(0,\cdot)  =0,\quad  \pd_tu(0,\cdot)  =0 & \mbox{in}\ M.
\end{array} \right.
\ee
Moreover, since $\A_1=\A_2$ on $(0,T)\times \pd M$ by assumption, we have $\pd_{\bar\nu} u=\pd_{\bar\nu} u_3-\pd_{\bar\nu} u_1=\Lambda_{\A_2,q_2}f-\Lambda_{\A_1,q_1}f=0$. Multiplying \eqref{eq3} by $u_2$ and integrating by parts we obtain
\bel{t1e} \int_{\M}[\A\nabla^{\g}u_1u_2+qu_1u_2]\,dV_{\bar{g}}=0,\ee
where $dV_{\bar{g}}(t,x)=|\bar{g}|^{\frac{1}{2}}dt \wedge dx$ denotes the volume form for $\bar{g} = -dt^2+g$. Applying \eqref{GO2}, we obtain
$$0=\lim_{\rho\to+\infty}\rho^{-1}\int_{\M}[\A\nabla^{\g}u_1u_2+qu_1u_2]\,dV_{\bar{g}} =i\lim_{\rho\to+\infty}\int_{\M} (\A\nabla^{\g}\phi) e^{-2\rho \Im(\phi)} v_1\bar{v}_2\,dV_{\bar{g}}.$$
\noindent Recall that the functions $v_1$ and $v_2$ are supported in a small tubular neighborhood of the null geodesic $\beta$. Thus the integrand in the above equation is supported near $\beta$ and we can the use Fermi coordinates $z=(z^0,\ldots,z^n)=(s,z')=(s,r,z^2,\ldots,z^n)=(s,r,z'')$ to compute the limit. We have:
$$ v_1\bar{v}_{2}= \rho^{\frac{n}{2}} v_{1,0}\bar{v}_{2,0} \chi^2(\frac{|z'|}{\delta})= \rho^{\frac{n}{2}} \det(|Y(s)|)^{-1} \mathcal{G}(s) \chi^2(\frac{|z'|}{\delta}),$$
where
$$\mathcal{G}(s)= e^{\frac{1}{2}\int_{s_-}^s (\A(\tau,0)\dot\beta)\,d\tau }.$$  
Note that:
$$  (\A\nabla^{\g}\phi)(s,0) \mathcal{G}(s) = 2 \dot{\mathcal{G}}(s).$$
Using Lemma~\ref{fermi} we have $ ||\det\bar{g}|(s,z')-1| \leq C|z'|^2$. Using the fact that $\A,q=0$ on the set $\hat{\M}\setminus \M$ and the above two obsevations we get
$$\lim_{\rho\to+\infty}\int_{a_0}^{b_0} \int_{|z'|<\delta} \rho^{\frac{n}{2}} \dot{\mathcal{G}}(s) e^{-2\rho \Im(\phi)} (\det |Y(s)|)^{-1} \chi^2(\frac{|z'|}{\delta}) \,ds \wedge dz'=0.$$
Lemma~\ref{ricB} implies that
$$ \int_{|z'|<\delta} \rho^{\frac{n}{2}}e^{-2\rho \Im(\phi)} (\det |Y(s)|)^{-1} \chi^2(\frac{|z'|}{\delta}) \,dz'=C + \mathcal{O}(|\rho|^{-\infty}),$$
where $C\neq0$ is a constant. Combining Fubini's theorem and the above equation we have
$$0=\int_{a_0}^{b_0} \dot{\mathcal{G}}(s) \,ds= \int_{s_-}^{s_+ } \dot{\mathcal{G}}(s) \,ds =\mathcal{G}(s_+) -\mathcal{G}(s_-).$$
Since $\mathcal{G}(s_-)=1$ by definition, we conclude that $ \mathcal{G}(s_+)=1$, which implies that for any null geodesic $\beta \in \mathcal{D}$, we have
$$ \exp(\frac{1}{2}\int_\beta \A)=1.$$
Since $\supp{\A} \subset \mathcal{E}$, we can conclude that:
\bel{DNdataforA}
\frac{1}{4\pi i} \int_\beta \A \in \mathbb{Z},
\ee
 for any maximal null geodesic in $\R \times M$. Now consider equation \eqref{DNdataforA} and write $\beta(t)=(\tilde{t}+t,\gamma(t))$, where $t \in I$ and $\tilde{t}\in\R$. We consider the family of null geodesics $\beta_{s}(t)=(s+t,\gamma(t))$ and note that equation \eqref{DNdataforA} holds for all $s \in\R$. Since $\A$ is of compact support, we must have that $\int_{\beta_s}\A$ vanishes for $|s|$ large. Together with the continuity of $\A$, we can conclude that it must vanish for all $s$. Therefore, for any maximal null geodesic in $\R\times M$ we have
$$\int_\beta \A  = 0.$$
Finally, under the hypothesis of Theorem~\ref{t1}, together with Proposition~\ref{introlightray}, we conclude that the first equation in \eqref{gauge} holds.
\subsection{Recovery of the light ray transform of $q$}
 The previous discussion yields that $\A_1-\A_2=\bar{d}\psi$ for some $\psi \in \CI^2(\M)$ with $\psi|_{\pd \M}=0$. We define $\tilde{\psi}=\frac{1}{2}\psi$, and let  
\bel{}
\begin{aligned}
\tilde{\A}_2(t,x)&=\A_2(t,x)+2 d\tilde{\psi}=\A_1(t,x) \quad \forall(t,x)\in \M,\\
\tilde{q}_2(t,x)&=q_2(t,x)+\Delta_{\g}\tilde{\psi} -\A_2\nabla^{\g}\tilde{\psi}-\left\langle \nabla^{\g}\tilde{\psi},\nabla^{\g}\tilde{\psi}\right\rangle_{\g}\quad \forall(t,x)\in \M.
\end{aligned}
\ee
The gauge invariance of the DN map implies that
\bel{t1aaa}
\Lambda_{\A_1,q_1}=\Lambda_{\tilde{\A}_2,\tilde{q}_2}=\Lambda_{\A_1,\tilde{q}_2}.
\ee
Furthermore, we also know that $\supp{(\tilde{q}_2 - q_1)} \subset \mathcal{E}$.
We now proceed as in the previous section. We start by fixing $q=q_1-\tilde{q}_2$. As before, we take $\beta$ to be a maximal null geodesic in $\mathcal{D} \subset \M$ and extend it to $\hat{\M}$. For $j=1,2$, we define 
$$u_j\in \mathcal C([0,T];H^1( M))\cap \mathcal C^1([0,T];L^2( M)),$$ 
as the Gaussian beam solutions of \eqref{Gsol} (with $\A_2$ replaced with $\tilde{\A}_2$ and $q_2$ replaced with $\tilde{q}_2$), taking the form \eqref{go1}-\eqref{go2}. We fix $u_3$ the solution of \eqref{eq1} with $\A=\tilde{\A}_2$, $q=\tilde{q}_2$ and $f$ given by
$$f(t,x):=u_1(t,x)=v_1(t,x)e^{i\rho \phi},\quad (t,x)\in(0,T)\times\pd M.$$ Then the function $u=u_3-u_1$ solves
\bel{eq3*}
\left\{ \begin{array}{ll}  -\Delta_{\g} u +\tilde{\A}_2(t,x)\nabla^{\g}u+ \tilde{q}_2(t,x) u  = qu_1, & \mbox{in}\ \M,\\  
u  =  0, & \mbox{on}\ (0,T)\times \partial M,\\  
 u(0,\cdot)  =0,\quad  \pd_tu(0,\cdot)  =0 & \mbox{in}\ M.
\end{array} \right.
\ee
Moreover, \eqref{t1aaa} implies $\pd_{\bar\nu} u=\pd_{\bar\nu} u_3-\pd_{\bar\nu} u_1=\Lambda_{\tilde{\A}_2,\tilde{q}_2}f-\Lambda_{\A_1,q_1}f=0$. Multiplying \eqref{eq3*} by $u_2$ and integrating by parts we obtain
\bel{t1eee}\int_{\M}qu_1u_2\,dV_{\g}=0.\ee
Applying \eqref{GO2}, we obtain
\bel{pr2}
\lim_{\rho \to \infty} \int_{\M}qe^{-2\rho\Im{\phi}}v_{1}\bar{v}_{2}\,dV_{\bar{g}}=0,
\ee
Here, we note that since $\tilde{\A}_2=\A_1$, we have
$$ v_1\bar{v}_{2}= \rho^{\frac{n}{2}} v_{1,0}\bar{v}_{2,0} \chi^2(\frac{|z'|}{\delta})= \rho^{\frac{n}{2}} \det(|Y(s)|)^{-1}\chi^2(\frac{|z'|}{\delta}).$$
Thus, taking the limit $\rho \to \infty$ and using Lemma~\ref{ricB} we deduce that
\bel{}
\int_{t \in I} q(\beta(t)) \,dt =0.
\ee
This equation only holds for maximal null geodesics in $\mathcal D$, but since $\supp{q} \subset \mathcal{E}$, we can conclude that this equation holds for all maximal null geodesics $\beta$ in $\R \times M$. Together with Proposition~\ref{introlightray}, we conclude that $q=0$, thus completing the proof of Theorem~\ref{t1}.


\section{Inversion of the light ray transforms}
\label{inversionsection}
This section is concerned with the proof of Proposition~\ref{introlightray}. We start with the inversion of the light ray transform of a scalar function $q$ satisfying $\supp q \subset \mathcal E$.
\begin{proof}[Proof of statement (i) in Proposition~\ref{introlightray}]
We know that for any maximal geodesic $\beta$ in $\R \times M$, 
$$ \int_{t\in I} q(\beta(t)) \,dt=0.$$ 
Using the identification of maximal null geodesics $\beta=(\tilde{t}+t,\gamma(t))$ with maximal geodesics in $M$, we conclude that
\bel{lightrayQ}
\mathcal{L}q(\tilde{t},\gamma)=\int_{I} q(\tilde{t}+t,\gamma(t))\,dt =0,
\ee
for all $\tilde{t} \in \R$ and all unit speed maximal geodesics $\gamma$ in $M$. Taking the Fourier transform of $\mathcal{L}(\tilde{t},\gamma)$ with respect to the variable $\tilde{t} \in \R$ and using the fact that $q$ is compactly supported, we deduce that
$$0=\widehat{\mathcal{L}q}(\tau,\gamma)=\int_I\int_\R e^{-i\tau\tilde{t}}q(t+\tilde{t},\gamma(t))\,d\tilde{t}\,dt=\int_I e^{i\tau t}\hat{q}(\tau,\gamma(t))\,dt,$$
where $\hat{q}(\tau,\cdot)$ denotes the Fourier transform of $q$ with respect to the time variable. Evaluating at $\tau=0$, we get
$$ \int_I \hat{q}(0,\gamma(t))\,dt=0,$$
for all geodesics $\gamma$ in $M$. Using Hypothesis~\ref{ginjectivity}, we deduce that $\hat{q}(0,\cdot)=0$. Next, we evaluate all the derivatives of $\widehat{\mathcal{L}}(\tau,\gamma)$ at $\tau=0$:
\bel{qderivatives}
0=[\pd^k_\tau \widehat{\mathcal{L}q}(\tau,\gamma)]|_{\tau=0}=[\pd^k_\tau \int_I e^{i\tau t} \hat{q}(\tau,\gamma(t))\,dt]|_{\tau=0}=\sum_{j=0}^k {{k}\choose{j}}\int_I (it)^{k-j}\pd^j_\tau\hat{q}(0,\gamma(t))\,dt.
\ee
We now use induction on $j$ to show that $\pd^j_\tau \hat{q}(0,\cdot)=0$ for all $j \in \N$. This is true for $j=0$ as shown above. Suppose the induction hypothesis holds for all $j<k$. Then using \eqref{qderivatives} implies that 
$$ \int_I \pd^k_\tau \hat{q}(0,\gamma(t)) \,dt=0.$$
Together with Hypothesis~\ref{ginjectivity}, we conclude that $\pd^k_\tau \hat{q}(0,\cdot)=0$. This completes the proof by induction. Now, since $q(t,\cdot)$ is a compactly supported function in $t$, we know that $\hat{q}(\tau,\cdot)$ is real analytic with respect to $\tau$ and since all the derivatives vanish at $\tau=0$, we conclude that $\hat{q}$ vanishes identically, which implies that $q \equiv 0$.
\end{proof}

It remains to prove the inversion of the light ray transform of a one-form $\mathcal A$ satisfying $\supp \A \subset \mathcal E$ up to the natural gauge. let us begin with some remarks and lemmas. 
We write $\A=b\,dt+\omega$. Similar to the proof above, we have
$$0=\widehat{\mathcal{L}\A}(\tau,\gamma)=\int_\R e^{-i\tau \tilde{t}} \int_I [b(\tilde{t}+t,\gamma(t))+\omega(\tilde{t}+t,\gamma(t))\dot{\gamma}(t)] \,dt \,d\tilde{t},$$
for all maximal geodesics $\gamma \in M$. 
Since $\A$ is compactly supported, we can interchange the order of integration to obtain 
$$0=\widehat{\mathcal{L}\A}(\tau,\gamma)= \int_I (\hat{b}(\tau,\gamma(t))+\hat{\omega}(\tau,\gamma(t))\dot{\gamma}(t))\,dt,$$
where $\hat{b}$ and $\hat{\omega}$ denote the Fourier transform in time of the compactly supported function $b$ and the one-form $\omega$ respectively. Evaluating at $\tau=0$, we obtain
$$ \mathcal{I}_{\gamma}( \hat{b}(0,\gamma) + \hat{\omega}(0,\gamma))=0.$$
Using Hypothesis~\ref{ginjectivity} together with smoothness properties of the Helmholtz decomposition discussed in Section~\ref{mainresult}, we deduce that
\bel{0injectivity}
\hat{b}(0,x)=0, \quad \hat{\omega}(0,x)=d\psi_0(x),
\ee
for some $\psi_0 \in \CI^1(M)$ vanishing on $\pd M$. Note that since $\omega$ is $\CI^1$ smooth, and since $d\psi_0(x)=\hat{\omega}(0,x) \in \CI^1(M;T^*M)$, it follows that
$$ \psi_0 \in \CI^2(M).$$  
We have the following lemma.
\begin{lem}
\label{sequence}
Let $\psi_{-1}:=0$ and $\psi_0$ be as above. There exists a sequence of functions $\{\psi_k\}_{k=1}^{\infty} \subset \CI^2(M)$ all vanishing on $\pd M$ and such that for every $k \geq 0$,
$$ \pd^k_\tau \hat{\omega}(0,x) = d\psi_k(x),\quad \text{and}\quad \pd^k_\tau \hat{b}(0,x)=ik\psi_{k-1}(x).$$
\end{lem}

\begin{proof}
We proceed by induction. The claim is true for $k=0$ as discussed above. Suppose it holds for all $m < k$. For $m=k$, we have:
$$\pd^{k}_{\tau}\widehat{\mathcal{L}\A}(0,\gamma)=\sum_{j=0}^k {{k}\choose{j}}\int_I (it)^{k-j}[\pd^{j}_\tau\hat{b}(0,\gamma(t))+\pd^{j}_\tau\hat{\omega}(0,\gamma(t))\dot{\gamma}(t)]\,dt=0.$$
Now using the induction hypothesis together we can simplify this expression to obtain
$$\int_{I}[\pd^k_\tau\hat{b}(0,\gamma(t))+\pd^k_\tau\hat{\omega}(0,\gamma(t))\dot{\gamma}(t)]\,dt=-\sum_{j=0}^{k-1}{{k}\choose{j}}(\int_I (it)^{k-j} [(ij)\psi_{j-1}+ d\psi_j\dot{\gamma}(t)]\,dt).$$
Now using the fact that $\int_{I} t^{k-j}d\psi_j\dot{\gamma}(t)\,dt=-\int_{I} (k-j)t^{k-j-1}\psi_j \,dt$, we can simplify the right hand side of the above equation to
\[
\begin{aligned}
&-\sum_{j=0}^{k-1} {{k}\choose{j}} (\int_I [(it)^{k-j} (ij) \psi_{j-1}-(it)^{k-j-1}i(k-j)\psi_j]\,dt)\\
&=-\sum_{j=0}^{k-1} {{k}\choose{j}}(\int_I [(it)^{k-j}(ij)\psi_{j-1}]\,dt)+\sum_{j=1}^{k} {{k}\choose{j}}(\int_I[(ij)(it)^{k-j}\psi_{j-1}]\,dt)\\
&=ik\int_I \psi_{k-1}(\gamma(t))\,dt.
\end{aligned}
\]
Hence we have
$$\int_{I}[\pd^k_\tau\hat{b}(0,\gamma(t))-ik\psi_{k-1}(\gamma(t))+\pd^k_\tau\hat{\omega}(0,\gamma(t))\dot{\gamma}(t)]\,dt=0.$$
Using Hypothesis~\ref{ginjectivity} together with smoothness properties of the Helmholtz decomposition in Section~\ref{mainresult}, we conclude that there exists $\psi_k \in \CI^1(M)$ vanishing on $\pd M$ such that
$$ \pd^k_\tau \hat{\omega}(0,x) = d\psi_k(x),\quad \text{and}\quad \pd^k_\tau \hat{b}(0,x)=ik\psi_{k-1}(x).$$
Since $d\psi_k= \pd^k_\tau \hat{\omega}(0,x) \in \CI^1(M;T^*M)$ and $\psi_k \in \CI^1(M)$, we conclude that
$$ \psi_k \in \CI^2(M).$$ 
This completes the induction argument.
\end{proof}

\begin{proof}[Proof of statement (ii) in Proposition~\ref{introlightray}]
\noindent Let us define
\bel{gaugedef}
\psi(t,x):=\int_0^t b(s,x) \,ds=\int_{-\infty}^{t} b(s,x) \,ds,
\ee
where we are using the fact that $\supp{\A} \subset \M$. Note that since $\A$ has compact support, equation \eqref{0injectivity} implies 
\bel{}
\psi(t,x)=\psi(T,x)=\int_{0}^T b(s,x) \, ds= \hat{b}(0,x)=0\quad \text{for $t \geq T$.}
\ee
Similarly, $\psi(t,x)=0$ for $t\leq 0$. Again, we let $\hat{\psi}(\tau,x)$ denote the Fourier transform of $\psi$ in the time variable. 
Note that since $\psi$ is compactly supported in time, $\hat{\psi}(\tau,x)$ is analytic with respect to $\tau$. Let us define the coefficients $\{\tilde{\psi}_k\}_{k=0}^{\infty}$ through
$$ \hat{\psi}(\tau,x)= \sum_{k=0}^{\infty} \frac{\tilde{\psi}_k(x)}{k!}\tau^k.$$
We claim that $ \tilde{\psi}_k(x) = \psi_k(x)$ holds for all $k \geq 0$ and all $x \in M$. To see this note that by definition $\pd_t \psi = b$. Hence $i\tau \hat{\psi}(\tau)=\hat{b}(\tau)$. Now differentiating this expression $k+1$ times and evaluating at $\tau=0$ we deduce that
$$i\tilde{\psi}_{k}=i \hat{\psi}^{(k)}(0)=i \frac{1}{k+1}\hat{b}^{(k+1)}(0)=i\psi_{k},$$
where we used Lemma~\ref{sequence} in the last step.  Thus, we have
$$ \hat{\psi}(\tau,x)=\sum_{k=0}^\infty \frac{\psi_k(x)}{k!}\tau^k.$$
Since $\omega(t,x)$ is also compactly supported in $t$, Lemma~\ref{sequence} implies that
\bel{infiniteseries}
\hat{\omega}(\tau,x)=\sum_{k=0}^\infty \frac{\pd_\tau^{k} \hat{\omega}(0,x)}{k!} \tau^k=\sum_{k=0}^\infty \frac{d\psi_k(x)}{k!} \tau^k.
\ee
Note that for every fixed $x \in M$, the following estimate holds:
$$ |\pd_\tau^k \hat{\omega}(0,x)| \leq \|\pd_\tau^k \hat{\omega}(\tau,x)\|_{L^{\infty}(\R)} \lesssim \int_0^T |t|^k |\omega(t,x)|\,dt,$$
which implies that the infinite series \eqref{infiniteseries} is uniformly convergent with respect to $x \in M$, and therefore we can write
$$\hat{\omega}(\tau,x)=\sum_{k=0}^\infty \frac{d\psi_k(x)}{k!} \tau^k=d(\sum_{k=0}^\infty \frac{\psi_k(x)}{k!} \tau^k)=d\hat{\psi}(\tau,x).$$
Hence, $\omega = d\psi$, and subsequently
\bel{}
\A = b \,dt + \omega = (\pd_t \psi) \,dt + d\psi = \bar{d}\psi.
\ee
Note that $\psi|_{(0,T)\times \pd M}=0$ as $\psi_k|_{\pd M}=0$ for all $k\geq 0$. We will now show that $\psi \in \CI^2(\M)$. Indeed, it is clear from equation \eqref{gaugedef} that
$\psi \in \CI^1(\M)$. But then since $\bar{d}\psi=\A \in \CI^1(\M;T^*\M)$ we may conclude that $\psi \in \CI^2(\M)$. 
\end{proof}

\subsection{Remarks}
With the proof of Theorem~\ref{t1} complete, let us state a few remarks. We start with the auxiliary condition that the coefficients $\A,q$ should be known on the set $\M \setminus \mathcal{E}$. This is merely an artifact of the light ray inversion method, as the Gaussian beam construction shows that the light ray transforms of $\A$ and $q$ can be obtained on the set $\mathcal D$. As the ratio $D_g(M)/T$ grows, the set $\mathcal E$ grows and the coefficients are therefore recovered in a large set that is closer in size to the optimal set $\mathcal D$. The assumption that $D_g(M)<\infty$ in particular implies that the manifold $(M,g)$ should be non-trapping. To illustrate this remark, consider $M=[0,1] \times S^1$ with $S^1$ denoting the unit circle and note that in this case $D_g(M)=\infty$. 


Regarding the smoothness of the metric $g$, note that we only require the metric in Fermi coordinates to be $\CI^4$ smooth. This however, can only be guaranteed if the metric is a priori known to be $\CI^6$ smooth, as there could be some loss of regularity in the angular directions of the exponential map of a Riemannian manifold. One can in fact use a mollification of the phase function $\phi$ to improve the result to $g \in \CI^4(M;\Sym^2M)$. The key here is that there is no loss of regularity in the direction tangent to the null geodesic in Fermi coordinates. We believe that the result could be extended to $g \in \CI^2(M;\Sym^2M)$, but this will require mollification of the metric, $g_\rho$. Improvements beyond the $\mathcal C^2$-smoothness for the metric $g$ could be much harder.

\section*{Acknowledgments}
A.F was supported by EPSRC grant EP/P01593X/1, J.I. was supported by the Academy of Finland (decision 295853), Y.K. was partially supported by the Agence Nationale de la Recherche grant ANR-17-CE40-0029 and L.O was supported by the EPSRC grants EP/P01593X/1 and EP/R002207/1.

\end{document}